\documentclass[reqno,11pt]{amsart}
\usepackage[margin=1in]{geometry}
\newcommand{\vvNumberWithin}{section}

\RequirePackage{amsthm}
\RequirePackage{amsmath}
\RequirePackage{amsfonts}
\RequirePackage{amssymb}
\usepackage[all]{xy}
\RequirePackage{tikz-cd}

\RequirePackage{txfonts}
\RequirePackage{lmodern}
\RequirePackage{enumerate}
\RequirePackage{hyperref}
\RequirePackage{xstring}

\newcommand*{\FIXME}[1]{}

\newcommand*{\Q}{\mathbb{Q}}

\newcommand*{\Z}{\mathbb{Z}}

\DeclareMathOperator{\id}{Id} 
\DeclareMathOperator{\Spec}{Spec}
\DeclareMathOperator{\Hom}{Hom}
\let\hom\relax
\DeclareMathOperator{\hom}{Hom}

\newtheoremstyle{misc}%
     {\topsep}
     {\topsep}
     {}
     {}
     {\itshape}
     {}
     { }
     {}

\newtheoremstyle{newdef}{\topsep}{\topsep}{}{}{\bfseries}{.}{ }{\thmnumber{#2}\thmnote{ #3}}

\newtheorem{master}{Master}[\vvNumberWithin]

\theoremstyle{newdef}

\swapnumbers
\theoremstyle{plain}
\newtheorem{theorem}[master]{Theorem}
\newtheorem*{theorem*}{Theorem}

\newtheorem*{result*}{Result}
\newtheorem{lemma}[master]{Lemma}
\newtheorem*{lemma*}{Lemma}

\newtheorem*{corollary*}{Corollary}
\newtheorem{proposition}[master]{Proposition}
\newtheorem*{proposition*}{Proposition}
\newtheorem{assumption-proposition}[master]{Assumption/Proposition}

\theoremstyle{definition}

\newtheorem*{example*}{Example}

\newtheorem*{application*}{Application}
\newtheorem{definition}[master]{Definition}
\newtheorem*{definition*}{Definition}
\newtheorem{remark}[master]{Remark}
\newtheorem*{remark*}{Remark}
\newtheorem{para}[master]{}
\newtheorem*{para*}{}

\newtheorem*{notation*}{Notation}

\newtheorem*{question*}{Question}

\newtheorem*{problem*}{Problem}

\theoremstyle{remark}

\newtheorem*{exercise*}{Exercise}

\makeatletter
\@addtoreset{master}{section}
\makeatother

\makeatletter
\@addtoreset{equation}{section}
\makeatother
\numberwithin{equation}{subsection}
 \renewcommand{\theequation}{%
      \ifnum\value{subsection} > 0
      \ifnum\value{subsubsection} > 0
      \thesubsubsection.\Alph{equation}%
      \else%
      \thesubsection.\Alph{equation}%
      \fi%
      \else%
      \thesection.\Alph{equation}%
      \fi%
    }
\newcommand{\nocontentsline}[3]{}
\newcommand{\tocless}[2]{\let\tempcontentsline=\addcontentsline\let\addcontentsline=\nocontentsline#1{#2}\hspace{-1em}\let\addcontentsline=\tempcontentsline}
\newcommand{\tocnonum}[2]{{\tocless#1{#2} \addcontentsline{toc}{subsection}{#2}}}

\newcommand*{\vvspan}[1]{{\langle #1 \rangle}}
\newcommand*{\vvtspan}[1]{{\langle #1 \rangle}^\infty}
\newcommand*{\iso}{\cong}

\title{Punctual gluing of $t$-structures and weight structures}
\author{Vaibhav~Vaish}
\address{Stat-Math Unit, Indian Statistical Institute, 8th Mile, Mysore Road, Bangalore 560059}
\email{vaibhav\_if@isibang.ac.in}

\begin{document}
\begin{abstract}
	We formulate a notion of ``punctual gluing'' of $t$-structures and weight structures. As our main application we show that the relative version of Ayoub's $1$-motivic $t$-structure restricts to compact motives.  We also demonstrate the utility of punctual gluing by recovering several constructions in literature. In particular we construct the weight structure on the category of motivic sheaves over any base $X$ and we also construct the relative Artin motive and the relative Picard motive of any variety $Y/X$.
\end{abstract}
\maketitle
\section{Introduction}

\tocless\subsection{} The notion of a $t$-structure on a triangulated category was introduced in \cite{BBD} to aid in construction of the (abelian) category of perverse sheaves. 
Any $t$-structure on a triangulated category $D$ automatically gives rise to an abelian category, it's ``heart'', and the problem of constructing the category of perverse sheaves then reduces to constructing an appropriate $t$-structure on the derived category of sheaves. To this end it is useful to produce new $t$-structures and \cite[\S 1.4]{BBD} lays down a procedure of gluing for the same.

	A related notion is that of a weight structure (sometimes called a co-$t$-structure), and here we follow \cite{bondarko_weights}. The procedure of gluing is available for weight structures as well \cite[\S 8.2]{bondarko_weights}.

\tocless\subsection{} The notion of ``punctual gluing'' of $t$-structures (or weight structures) is implicit, for example, in the punctual criterion of perversity (see \cite[\S 4.0]{BBD}) or weights (\cite[\S 5.1]{BBD}) on $D^b_m(X,\Q_l)$. In fact punctual gluing is merely a formalization of a very simple generalization of the process that went into the construction of perverse $t$-structures in \cite[\S 2]{BBD}. We briefly explain this below.

Let us recall the construction of the perverse $t$-structure on $D^b(X):=D^b_m(X,\Q_l)$ in \cite{BBD}. 
The standard gluing results of $t$-structures in \cite[\S 1.4]{BBD} involves constructing a $t$-structure knowing one on each stratum of a \emph{fixed} stratification of $X$ in locally closed subsets. 
Firstly, one fixes a monotone step function, the perversity function $p$. This fixes a $t$-structure on $D^b(S)$ for any locally closed subset $S\subset X$ depending only on $p(\dim S)$, denoted $t[{p(\dim X)}]$ (here $t$ is the standard $t$-structure). Then given any stratification $\mathcal S$ of $X$ one obtains a glued $t$-structure $p_{\mathcal S}$ on $D^b(X)$ using the gluing results of \cite[\S 1.4]{BBD}. Furthermore, by restricting the class of sheaves to so called ``constructible sheaves'' $D^b_c(X)$ (see \cite[\S 2.2]{BBD}), one ensures that the truncation for $p_{\mathcal S}$ on any specific constructible sheaf is independent of $\mathcal S$, provided we choose $\mathcal S$ fine enough. This truncation then corresponds to a $t$-structure on $D^b_c(X)$ which is defined to be the perverse $t$-structure corresponding to perversity $p$, $(D^{t\le p}(X), D^{t> p}(X))$. 

In \cite[\S 4.0]{BBD} one proves a punctual criterion for characterizing the perverse $t$-structure, that is, it characterizes the objects $A\in D^{t\le p}(X)$ by conditions on the stalk $A_x$ for all $x\in X$: $A\in D^{t\le p}(X)$ if and only if  $A_x\in D^{t\le 0}(x)[p(\dim x)] = D^{t\le -p(\dim x)}(x)$ where $D^{t\le 0}(x)$ is the negative part of the standard $t$-structure $t$ on $D^b(x)$.

The idea of ``punctual gluing'' is to be able to turn this characterization into a definition and construct a $t$-structure on (the derived category of sheaves on) a scheme $X$, knowing one on each point $x\in X$ (not necessarily closed). It turns out that the key property of the category of constructible sheaves in this context is ``continuity'' that is:
	\begin{equation}\label{2lim}
		D^b_c(x) \cong 2\lim_{x\in U}D^b_c(U)
	\end{equation}
	for any $x\in X$ where $U$ varies over open neighborhoods of $x$, where for $f:V\hookrightarrow U$, the corresponding map in the limit is $f^*:D^b_c(U)\rightarrow D^b_c(V)$. The restrictions on perversity function $p$ can then be recast as restrictions on the corresponding $t$-structure $t[p(\dim x)]$ on $D^b_c(x)$, leading to the notion of \emph{continuity of $t$-structures} \ref{gluing:continuityForT}.
	
	Then, in presence of continuity and given a $t$-structure $t(x)$ each on $D^b_c(x)$ satisfying continuity of $t$-structures it will be formal to construct a glued $t$-structure on $D^b_c(X)$:
	\begin{proposition*}[See \ref{gluing:mainresult}]
		Fix a Noetherian scheme $X$, and assume that for each $S\subset X$ we are given a triangulated category $D_S$ (``constructible sheaves'') with Grothendieck's four functors \ref{gluing:4functors} and for each Zariski point $x\in X$ (not necessarily closed) we are given a triangulated category $D(x)$ satisfying continuity \ref{gluing:continuity}. 
		
		Assume furthermore that for each $x\in X$ we are given a $t$-structure (resp. a  weight structure) $(D^{\le}(x), D^>(x))$ on $D(x)$ satisfying continuity for $t$-structures (equivalently, for weight structures) \ref{gluing:continuityForT}. Then:
		\begin{align*}
			D^{\le}(X)	:=\{a&\in D_X\big| \epsilon^*(a)\in D^{\le}(x)\text{ for }\epsilon:x\hookrightarrow U\text{ any point of }X\} \\
			D^{>}(X)	:=\{a&\in D_X\big| \epsilon^!(a)\in D^{>}(x)\text{ for }\epsilon:x\hookrightarrow U\text{ any point of }X\}		
		\end{align*}
		is a $t$-structure (resp. a weight structure) on $D_X$.
	\end{proposition*}
	
	Here $\epsilon^*(-)$ is the map relating $D^b_c(U)$ for $U=X$ and $D^b(x)$ in the equation \eqref{2lim} while $\epsilon^! := {\epsilon'}^* i^!$ with $\epsilon:x\overset{\epsilon'}\hookrightarrow Y\overset{i}\hookrightarrow X$ with $x$ generic in $Y$. \-\\

	Here $D_X$ can be bounded derived category of mixed $l$-adic sheaves $D^b(X,\Q_l)$, but it can also be the tensor triangulated category of motivic sheaves. This allows us to construct $t$-structures in the ``relative'' setting (that is over some scheme $X$) beginning with those in the absolute setting (that is when $X=\Spec k$ for a field $k$).

\tocless\subsection{} Our interest in punctual gluing lies in the quest to construct new $t$-structures on the (triangulated) category of motives. 

	In the conjectural picture of Beilinson \cite{beilinson1987height} one should be able to construct an abelian category of mixed motivic sheaves over any scheme $X$, related through the formalism of Grothendieck's six operations, and such that the relation with algebraic cycle homology theories is as expected. While this category still seems distant, we now have a tensor triangulated category, denoted here by $DM(X)$, which acts as the derived category of this conjectural category, is equipped with the  Grothendieck's six functors (we work with $\Q$-coefficients throughout), and has the correct relationship with the algebraic cycle homology theories -- this is now due to several people, for example, this can be the category of etale motivic sheaves without transfer, $DA(X,\Q)$ of Ayoub \cite[2.1]{ayoub2012relative} or Beilinson motives, $DM_{B, c}(X,\Q)$ of Cisinski-Deglise \cite{cisinski2012triangulated}. 
	
	The Beilinson's conjectures then reduce to the problem of constructing an appropriate $t$-structure on $DM(X)$. While this $t$-structure seems to be out of reach at the moment, it is possible to construct appropriate subcategories of $DM(X)$ on which this conjectural $t$-structure restricts, and also explicitly compute the corresponding restriction. A step in this direction is taken in \cite{ayoub2011nmotivic} where he constructs this restriction on so called categories of $n$-motives, for $n=0, 1, 2$. For $n=0, 1$, these categories (but not the $t$-structure) are also explored in \cite{ayoub20091motivic}. For $n=1$, this $t$-structure is also constructed in \cite{MR2102056}, and in fact was already indicated by Voevodsky \cite{voevodsky}.
	
	The main result of \cite{lehalleur2015motivic} is to relativize this to any base $S$. However, as a limitation of their work, they are unable to prove that this $t$-structure restricts to compact objects, which is important for it to make geometric sense. 
	By re-constructing their $t$-structure using punctual gluing, we are able to overcome this limitation and we prove the following:
	\begin{theorem*}[See \ref{application:t1MMrespectsCompact}]
		The $t$-structure $t^1_{MM}(S)$ on $DA^1(S)$ of \cite[\S 4]{lehalleur2015motivic} restricts to compact objects, that is a $t$-structure on $DA^1_c(S)=DM^{coh}_{1-mot}(S)$.
	\end{theorem*}

\tocless\subsection{} We demonstrate the utility of punctual gluing by recovering several constructions in the literature in a more streamlined fashion. For example we recover the weight structures on motivic sheaves due to \cite{hebert2011structure} and \cite{bondarko2014weights}:
\begin{theorem*}[See \ref{application:weights}]
	The weight structure due to \cite{bondarko_weights} on $DM(k)$ for $k$ any field satisfies continuity of weight structures, and hence by punctual gluing one can construct a weight structure $(DM_{w\le 0}(X), DM_{w>0}(X))$ on $DM(X)$, for any Noetherian scheme $X$.
\end{theorem*}

In \cite{ayoub20091motivic} Ayoub-Viale define the triangulated subcategory of $0$-motives and $1$-motives inside $DM^{eff}(k)$. Dually, one can define the subcategory of $0$-motives and $1$-motives inside $DM^{coh}(k)$, and this is denoted here by $DM^{coh}_{i-mot}(k)$ for $i\in\{0, 1\}$. Then it follows from Ayoub-Viale \cite[\S 5]{ayoub20091motivic} (and using duality, see \cite[3.12]{lehalleur2015motivic} for example) that the inclusion $DM^{coh}_{i-mot}(k)\hookrightarrow DM^{coh}(k)$ admits a right adjoint. As another example this can be immediately relativized using punctual gluing, recovering the main construction of \cite[\S 3]{lehalleur2015motivic}:
\begin{theorem*}[See \ref{application:picard}] Let $i\in\{0,1\}$. Then
	\begin{align*}
		DM^{coh}_{i-mot}(X)\hookrightarrow DM^{coh}(X)\text{ admits a right adjoint }\omega^i_X:DM^{coh}(X)\rightarrow DM^{coh}_{i-mot}(X)		
	\end{align*}
\end{theorem*}
	The object $\omega^0_X(p_*1_Y)$ (resp. $\omega^1_X(p_*1_Y)$) for $p:Y\rightarrow X$ computes the \emph{relative Artin motive} (resp. \emph{relative Picard motive}) of $Y$ over $X$. In particular the functor $\omega^0$ already appeared in \cite{ayoub2012relative} and in \cite{vaish2016motivic}.

\tocless\subsection{} 	There are several other conjectural $t$-structures on $DM(X)$ which are of interest. For example, one should be able to construct analogue of Morel's $t$-structures \cite{morelThesis} (rather, their mild generalization \cite{arvindVaish}). These $t$-structures depend on a choice of a monotone step function $D$, the weight profile. The constructions for $D=Id$, the identity function, was already used to recover the motive of reductive Borel-Serre compactification for an arbitrary Shimura variety in \cite{vaish2016motivic}.

	Punctual gluing will be used in \cite{vaish2017weight} to recover analogue of Morel's $t$-structures for some other weight profiles which in turn will have applications -- we will be able to construct a motivic analogue of the intersection complex for any arbitrary three-fold over any field $k$ which will be canonical enough to admit actions from correspondences on the open part (for example, Hecke actions on Shimura threefolds). We will also recover several invariants of singularities motivically. 
	
	The article \cite{vaish2017weight} was our main motivation for the formulation of punctual gluing.

\tocnonum\subsection{Outline} In \S\ref{sec:prelims} we review the notion of $t$-structures and weight structures, and recall some well known results. In \S\ref{sec:gluing}, we formulate the notion of continuity for $t$-structures and weight structures \ref{gluing:continuityForT} and prove our main result on punctual gluing \ref{gluing:mainresult}. \S\ref{sec:motives} recalls the notion of motivic sheaves \ref{motives:6functors} and that of cohomological motives \ref{definieCohMotives}, the arena where our applications live. In \S\ref{sec:applications} we have our main applications. 

In \S\ref{sec:application:bondarko} we construct weight structures on motivic sheaves \ref{application:weights}, beginning with their absolute counterpart. In \S\ref{sec:application:artinpicard}  we construct the right adjoints $\omega^i_X:DM^{coh}(X)\rightarrow DM_{i-mot}^{coh}(X)$ ($i\in\{0,1\}$) to the inclusion of the $i$-motivic category $DM_{i-mot}^{coh}(X)\subset DM^{coh}(X)$, thereby constructing the relative Artin motive and the relative Picard motive. For $i=0$, the proof is self contained but for $i=1$ we need a Hom-vanishing computed in \cite{vaish2017weight}. We can substitute the results of this section by those in \cite[\S 3]{lehalleur2015motivic} and hence the main application \S\ref{sec:application:1motivic} can be seen independently of this section.

   Finally in \S\ref{sec:application:1motivic} we have our main application \ref{application:t1MMrespectsCompact} demonstrating that the $t$-structure $t^1_{MM}(S)$ of  \cite{lehalleur2015motivic} restricts to compact objects which we show by reconstructing the same via punctual gluing.
	

\section{$t$-structures and weight structures}\label{sec:prelims}
\begin{para}A $t$-structure on a triangulated category $D$ (see \cite{BBD}) is a pair of full subcategories $(D^{\le t}, D^{>t})$ satisfying three properties:
\begin{itemize}
\item (Orthogonality) $\hom (a,b)=0, \forall a\in D^{\le t}, b\in D^{>t}$
\item (Invariance) $D^{\le t}[1]\subset D^{\le t}$, and $D^{> t}[-1]\subset D^{>t}$
\item (Decomposition) $\forall c\in D$, there is a distinguished triangle $a\rightarrow c\rightarrow b\rightarrow $ with $a\in D^{\le t}$ and $b\in D^{>t}$. 
\end{itemize}
\end{para}
\begin{para}
A weight structure (also called a a co-$t$-structure) on a triangulated category $D$ (see \cite{bondarko_weights}, for example) is the same as a $t$-structure, except that instead of invariance, it satisfies co-invariance. Also, one needs to assume that the given subcategories are closed under taking summands in $D$ (this is automatic for $t$-structures). That is a weight structure is a pair of full subcategories $(D^{\le t}, D^{>t})$ satisfying:
\begin{itemize}
	\item (Karaubi-closed) $D^{\le t}$ and $D^{>t}$ are closed under taking summands.
	\item (Orthogonality) $\hom (a,b)=0, \forall a\in D^{\le t}, b\in D^{>t}$
	\item (Co-invariance) $D^{\le t}[-1]\subset D^{\le t}$, and $D^{> t}[1]\subset D^{>t}$
	\item (Decomposition) $\forall c\in D$, there is a distinguished triangle $a\rightarrow c\rightarrow b\rightarrow $ with $a\in D^{\le t}$ and $b\in D^{>t}$. 	
\end{itemize}
\end{para}
\begin{para}
We will sometimes refer to $D^{\le t}$ as the negative part and $D^{>t}$ as the positive part of the $t$-structure (resp. of the weight structure).
\end{para}
\begin{para} For a $t$-structure, the decomposition of $c$ in the last property can be shown to be unique and then $a$ (resp. $b$) is often denoted as $\tau_{\le t}c$ (resp. $\tau_{>t}c$). It is easy to show that $\tau_{\le t}$ (resp. $\tau_{>t}$) give rise to a right (resp. left) adjoint to the fully faithful inclusion $D^{\le t}\hookrightarrow D$ (resp. $D^{>t}\hookrightarrow D$).

For a weight structure, the decomposition of an object $c$ is no longer unique. However we will continue to use the notations $\tau_{\le t}c$ (resp. $\tau_{>t}c$) to denote \emph{any} such decomposition.
\end{para}

\begin{para}We will be frequently interested in constructing $t$-structures which are also weight structures, that is where $D^{\le t}$ and $D^{>t}$ are triangulated subcategories of $D$ (in particular the core of the $t$-structure or the weight structure, $D^{\le t}\cap D^{>t}[1]$ resp. $D^{\le t}\cap D^{>t}[-1]$ by definition, is trivial). Let us call such a pair $(D^{\le t}, D^{>t})$ as an $m$-structure ($m$ stands for Morel, since the constructions are motivated by those used in \cite[\S 3]{morelThesis}).\end{para}

To construct a $m$-structure is equivalent to give an adjoint to a full triangulated subcategory:
\begin{proposition}\label{tIsAdjoint}
	Let $D^{\le t}$ be a full triangulated subcategory of a triangulated category $D$. Then to give a $m$-structure $(D^{\le t}, D^{> t})$ on $D$ is to give a right adjoint $\tau_{\le t}$ to the inclusion $D^{\le t}\hookrightarrow D$. In this case, $D^{> t}$ can be identified to the subcategory of objects $x$ in $D$ such that $\tau_{\le t}(x)=0$.
\end{proposition}
\begin{proof}
	Given the $m$-structure, $\tau_{\le t}$ is the truncation to the negative part, which is the right adjoint to the natural inclusion since this is a $t$-structure. Conversely, given a right adjoint $\tau_{\le t}:D\rightarrow D^{\le t}$, let $D^{> t}:=\{x\big|\tau_{\le t}(x)=0\}$. Then orthogonality, and invariance are obvious. As for decomposition, for any $h$, let $\tau_{\le t}h\rightarrow h$ denote the natural morphism of adjunction, and let $\tau_{>t}h$ be defined as the third part of the cone. Then it is enough to verify that $\tau_{\le t}(\tau_{>t}h)=0$. But the functor $\tau_{\le t}$ is triangulated since it is adjoint to a triangulated functor \cite[5.3.6]{neeman2014triangulated}, and applying it to the triangle, $\tau_{\le t}h\rightarrow h\rightarrow \tau_{>t}h\rightarrow $, the first map reduces to identity, and we see that $\tau_{\le t}(\tau_{>t}h)=0$ as required.
\end{proof}%
\begin{remark}
	It is possible to work with a $t$-structure instead of an $m$-structure for the previous proposition.
\end{remark}

\begin{para}
	A frequent approach to constructing a right adjoint to a triangulated functor $i:D^{\le t}\rightarrow D$ is to use a form of Brown representability -- if the category $D^{\le t}$ is compactly generated and $D$ has small direct sums, then the existence of adjoint is automatic \cite[8.4.4.]{neeman2014triangulated}. However this adjoint is unwieldy, and the real hard work is to show that it preserves compact objects (which, in the motivic categories, correspond to the geometric or constructible objects).

	Instead, in our approach, we focus on triangulated categories which may not have arbitrary direct sums. Punctual gluing still works, and we automatically get a $t$-structure and hence an adjoint functor. This has the added advantage that constructions are somewhat explicit and have good properties with respect to the compact objects.  It is not any weaker than the other approach - this is the content of \ref{compactToNonCompact} and \ref{compactAreGenerated}, and is explained below.
\end{para}

%

Following definitions will be useful:
\begin{definition}
	Let $D$ be a pseudo-abelian triangulated category. Let $A, B, H\subset D$ be full subcategories containing $0$ which are isomorphism closed -- that is if $x\cong y$ for some $x\in D$ and $y\in H$ (resp.  $y\in A$, resp. $y\in B$) then $x\in H$ (resp. $x\in A$, resp. in $x\in B$).
	Then we define the following full subcategories:
	\begin{align*}		
		Ext^1(B,A)	&:= \{ d\in D \big|\exists a\rightarrow d\rightarrow b\rightarrow  \text{ distinguished }\text{with }a\in A,b\in B\} \hspace{-40em}\\
	Ext^1(H) &:= Ext^1(H,H)	& 
		Ext^n(H)	&:= Ext^1(Ext^{n-1}(H))\;\forall n\ge 2\\
	Ext(H)&:=\cup_{i\ge 1}{Ext^i(H)} &
		H^\infty	&:= \{\oplus_{i\in I} h_i\big|h_i\in H, I\text{ arbitrary small }\}\\	
	Ext^\infty (H) &:=Ext(H^\infty)  &
		H[\Z] &:= \{ h[n]\big| h\in H,n\in \Z\}\\
	\mathcal R(H)&:=\{d\in D \big| d\text{ is a retract of some }h\in H\} \hspace{-40em}\\
	\vvspan{H} &:= \mathcal R(Ext(H[\Z])) &
		\vvtspan{H} &:= Ext^\infty(H[\Z])\\
		A^\perp &:= \{d\in D\mid \hom(a,d)=0,\forall a\in A\} & ^\perp A &:= \{d\in D\mid \hom(d,a)=0,\forall a\in A\}		
	\end{align*}
	(Here $H^\infty, Ext^\infty, \vvtspan{H}$ are defined when $D$ has small direct sums). Also say:
	\begin{align*}
		&\text{$A\perp B$ if $A^\perp\supset B \Leftrightarrow A\subset {^\perp B}$.} &  &\text{$(A,B)$ \emph{decompose} $H$ if $H\subset Ext^1(B,A)$.}\\
		&\text{$D$ is \emph{finitely generated by} $H$ if $D=\vvspan H $.} & &\text{$D$ is \emph{generated by} $H$ if $D=\vvtspan H $.}
	\end{align*} 
\end{definition}


\begin{remark}
	In some places in literature (e.g. in \cite{hebert2011structure}), $Ext^\infty(H)$ is defined as the full subcategory $Ext(Ext(H)^\infty)$. The two definitions are equivalent: $Ext(H)^\infty \subset Ext(H^\infty)$ since an arbitrary co-product of distinguished triangles is a distinguished triangle, and hence $Ext(Ext(H)^\infty) \subset Ext(Ext(H^\infty))=Ext(H^\infty)$. $Ext(H^\infty)\subset Ext(Ext(H)^\infty)$ since $H\subset Ext(H)$, giving the other inclusion. 
\end{remark}
\begin{remark}	
 	$\mathcal R(\vvtspan H)=\vvtspan H$ because the image of a projector can be constructed using homotopy colimits (see \cite[1.6.8]{neeman2014triangulated}) and $\vvtspan H $ is closed under countable direct sums and cones, hence under homotopy colimits.
\end{remark}
\begin{remark}\label{otherIsPerp}
	For both $t$-structures and weight structures $(D^{\le t})^\perp = D^{> t}$, and $D^{\le t} = {^\perp (D^{>t})}$ -- for $t$-structures this is the calculation in the proof of \ref{tIsAdjoint}, while for weight structures see \cite[1.3.3]{bondarko_weights}.
\end{remark}

If we can construct a $m$-structure on $\vvspan H$ we can do so on $\vvtspan H$, provided $H$ consists of compact objects:
\begin{proposition}\label{compactToNonCompact}
		Let $(A,B)$ form a $m$-structure on $H$ and assume that $A$ consists of compact objects. Then $(Ext^\infty(A), Ext^\infty(B))$ forms a $m$-structure on $Ext^\infty(H)$.
\end{proposition}
\begin{proof}
	Since $A$ consists of compact objects $A\perp B\Rightarrow A \perp B^\infty$. It follows that $A^\infty \perp B^\infty$. It is easy to see that $A'\perp B'\Rightarrow  Ext^1(A')\perp B'$ as well as $A'\perp Ext^1(B')$, and hence inductively, $Ext^n(A')\perp Ext^n(B')$. Since $Ext^n(-)\subset Ext^{n+1}(-)$, inductively, $Ext(A')\perp Ext(B')$, and applying to $A' = A^\infty$ and $B' = B^\infty$, we see the orthogonality of $(Ext^\infty(A), Ext^\infty(B))$. 
	
	Invariance is obvious since $Ext(-)$ and $(-)^\infty$ commute with shifts.
	
	For decomposition, note that $H\subset Ext^1(B,A)$, then $H^\infty \subset Ext^1(B^\infty, A^\infty)$, since arbitrary direct sum of distinguished triangles is a distinguished triangle. Also, if $A\perp B[1]$ and $H\subset Ext^1(B,A)$, then $Ext^1(H) \subset Ext^1(Ext^1(B), Ext^1(A))$ as is easily seen by the following diagram:
\[
\xymatrix{
a\ar@{..>}[d]\ar@{..>}[r] & h\ar[d]\ar@{..>}[r] & b\ar@{..>}[d]\ar@{..>}[r] & a[1]\ar@{..>}[d]\\
a_2\ar[r]\ar@{-->}[d] & h_2\ar[r]\ar[d] & b_2\ar[r]\ar@{-->}[d] & b[1]\ar@{-->}[d]\\
a_1[1]\ar[r] & h_1[1]\ar[r] & b_1[1]\ar[r] & a[2]
}
\]
	Here (bottom) dashed arrows exist, since $\hom(a_2,b_1[1]) =0$. Defining $a$ by the triangle $a\rightarrow a_2\rightarrow a_1[1]\rightarrow $ and similarly for $b$, we complete the diagram using dotted arrows. Hence $a\in Ext^1(A)$, $b\in Ext^1(B)$, and so $Ext^1(H)$ is decomposed by $(Ext^{1}(A),Ext^{1}(B))$. 
	
	Now claim for $Ext^n$, hence also $Ext$, follows by induction since $A\perp B[1]$ implies $Ext^1(A)\perp Ext^1(B)[1]\subset Ext^1(B[1])$. Applying to $A^\infty$ and $B^\infty$ we get the decomposition as required.

\end{proof}
\begin{remark}
	It is enough to work with a weight structure instead of an $m$-structure for the previous proposition.
\end{remark}
In the motivic setting, it is standard to lay down results in the situation of (subcategories) of the form $\vvtspan S$ where $S$ consists of compact objects. Then the claim that an $m$-structure restricts to compact objects is equivalent to saying that we have a $m$-structure on $\vvspan S$ by the previous proposition and the following:
\begin{proposition}\label{compactAreGenerated}
	Assume that $D=\vvtspan S$ and that the objects in $S$ are compact. Then the full subcategory of compact objects in $D$ can be identified with $\vvspan S$.
\end{proposition}
\begin{proof}
	This follows immediately from \cite[4.4.5]{neeman2014triangulated}.
\end{proof}

\section{Punctual Gluing}\label{sec:gluing}%
 A general procedure of gluing $t$-structures, which also works for co-$t$-structures, was laid down in \cite{BBD} -- if $D(X)$ denotes the (bounded) triangulated category of sheaves on $X$, then one can construct a  $t$-structure on $D(X)$ given a $t$-structure each on $D(Z)$ and $D(U)$ where $Z\subset X$ is closed and $U$ is its open complement. This can be generalized to yield gluing on any fixed stratification by induction. Under fairly mild conditions (``constructibility'' of objects under consideration), one can construct a stratification independent version of it (which was implicitly used in \emph{loc.~cit.} to construct the perverse $t$-structure). 

While constructibility is usually formulated as property of an individual sheaf, it is better seen to manifest as ``continuity'', a property of the category of constructible sheaves. Through a simple argument, it is possible to show that, in presence of continuity, to define a $t$-structure on $D(X)$ one only needs a $t$-structure on the triangulated categories $D(x)$ for $x$ any (Zariski) point of $X$. This is made precise below.

We begin with the following situation, which is standard:
	\begin{definition}[Grothendieck's four functors]\label{gluing:4functors}
		Given a scheme $X$ let $Sub(X)$ be the category of sub-schemes of $X$ (i.e.. schemes $f:Y\rightarrow X$, $f$ an locally-closed immersion as objects and obvious morphisms). In particular, all morphisms in $Sub(X)$ are immersions, of finite type. 
		
		We say \emph{formalism of Grothendieck's four functors} holds if:
		\begin{enumerate}[i.]
			 \item For any $Y$ in $Sub(X)$ we are given a triangulated category $D_Y$. 
			 \item For any morphism $f:Y\rightarrow Z$ in $Sub(X)$, we are given adjoint functors
			 	\begin{align*}
					f^*: D_Y \leftrightarrows D_Z : f_* & &\text{ and }& &
					f_!: D_Y \leftrightarrows D_Z : f^!
				\end{align*}
				such that there are isomorphisms of functors:
				\begin{align*}
					(f g)_* \iso f_*g_* & & (f g)_! \iso f_!g_! & & (f g)^* \iso g^*f^* & & (f g)^! \iso g^!f^!
				\end{align*}
				Also if $1:X\rightarrow X$ is the identity morphism, and $\id:D_X\rightarrow D_X$ denotes the identity natural transformation. Then we must have an isomorphism of functors:
				\begin{align*}
					1_* \iso 1_! \iso 1^* \iso 1^! \iso \id
				\end{align*}
				
			\item We are given natural transformations
				\begin{align*}
					f^!	\rightarrow f^* & & f_!	\rightarrow f_*
				\end{align*}
				where the first is an isomorphism for $f$ an open immersion, and the second an isomorphism for $f$ a closed immersion.
			\item If $i:Z\rightarrow Y$ is a closed immersion with $j:U\rightarrow Y$ the open immersion of the complement, then the functors $(i^*,i_*,i^!),(j_!,j^*,j_*)$ follow the formalism of gluing \ref{gluing:openClosed}.
		\end{enumerate}
	\end{definition}
	\begin{definition}[Formalism of gluing]\label{gluing:openClosed}
		The triangulated functors:
			\begin{align*}
				j_*, j_!: D_U	\leftrightarrows D_X	: j^* & & i_*: D_Z	\leftrightarrows D_X	: i^*, i^!
			\end{align*}
		between triangulated categories $D_U,D_Z,D_X$ are said to satisfy \emph{formalism of gluing} if:
		\begin{enumerate}[i.]
			\item $(j_!, j^*)$, $(j^*, j_*)$, $(i^*, i_*)$ and $(i_*, i^!)$ are adjoints tuples.
			\item $j^*i_*=0$ and hence by adjunction $i^*j_! = 0$ as well as $i^!j_* = 0$.
			\item (Localization)  $\forall A\in D_X$, we have functorial distinguished triangles:
			\begin{align*} j_!j^*A\rightarrow A\rightarrow i_*i^*A\rightarrow & & i_*i^!A\rightarrow A\rightarrow j_* j^*A\rightarrow \end{align*}
			where the morphisms come from adjunction.
			\item We have isomorphism of functors
			\begin{align*} \forall A\in D_U,\;j^*j_*A\xrightarrow\cong A\xrightarrow\cong j^*j_!A & & \forall B\in D_Z,\;i^*i_* B\xrightarrow\cong B\xrightarrow\cong i^!i_*B  \end{align*}
			with the morphisms coming from adjunction.
		\end{enumerate}
	\end{definition}
	We will need $f^!$ and $f^*$ defined not only for locally closed immersions, but also when $f:\Spec K\hookrightarrow X$ denotes any (Zariski) point of $X$ (that is $K$ is the residue field of $x\in X$, and we consider the induced map $\Spec K\rightarrow X$). While, in the situations we intend to apply this, $f^*$ is a given, $f^!$ needs to be defined by hand. We make this precise below:
	\begin{definition}[Extended formalism of gluing]\label{gluing:extended4functors}
		Let $X$ be a scheme and let us assume that Grothendieck's four functors \ref{gluing:4functors} exist. $f_Y:Y\rightarrow X$ will denote the natural immersion for any $Y\in Sub(X)$. The situation is said to satisfy \emph{extended formalism of gluing} if the following happens:
		
		Assume that for each Zariski point $\Spec k\hookrightarrow X$ we are given a triangulated category $D(k)$. Let $y=\Spec k$ denote the corresponding point in $X$ and let $Y=\overline {\{y\}}$. Let $Z\in Sub(X)$ be such that $y\in Z$ as well. Let $\epsilon_Z:\Spec k\hookrightarrow Z$ denote the natural morphism. Assume that in such a situation we are given a functor
			\begin{align*}
				\epsilon_Z^*:D_Z\rightarrow D(k)&  &\text{ such that }\epsilon_Z^* = \epsilon_{Z'}^{*}\circ f^*\text{ for all }\Spec k\xrightarrow{\epsilon_{Z'}} Z'\xrightarrow f Z\text{ factoring }\epsilon_Z
			\end{align*}

		Since $\overline{\{y\}} = Y$ therefore $Y\subset \bar Z$ and $i:Y\cap Z\hookrightarrow Z$ is a closed immersion. Let $\epsilon_{Y\cap Z}:\Spec k\rightarrow Y\cap Z$ be the natural morphism. Then we define:
		\[
			\epsilon_Z^! := \epsilon_{Y\cap Z}^*\circ i^!: D_Z\rightarrow D_{Y\cap Z}\rightarrow D(k)
		\]
		Notice that if $\Spec k\xrightarrow{\epsilon_{Z'}} Z'\xrightarrow f Z$ factors $\epsilon_Z$, with $i':Y\cap Z'\hookrightarrow Z'$ being the immersion, then $Y\cap Z\subset Y=\bar y$ lies in closure of $Y\cap Z'$, since $y\in Y\cap Z'$. Since $Y'\cap Z\subset Y\cap Z$ is also locally closed, it must be open in $Y\cap Z$. It follows that if $j:Y\cap Z'\rightarrow Y\cap Z$ denote the immersion, $j^*=j^!$, and hence we have a canonical identification of functors: 
		\[
			\epsilon_Z^! = \epsilon_{Y\cap Z}^*\circ i^! = (\epsilon_{Y\cap Z'}^*\circ j^*)\circ i^! = \epsilon_{Y\cap Z'}^*\circ (i\circ j)^! =\epsilon_{Y\cap Z'}^*\circ (f\circ i')^!=\epsilon_{Y\cap Z'}^*\circ i'^!\circ f^!=\epsilon_{Z'}^!\circ f^!
		\]
	\end{definition}
	
	Finally, we make the essential definitions:
	\begin{definition}[Continuity]\label{gluing:continuity}
		Assume the situation in \ref{gluing:extended4functors}.  Let $\epsilon_Y:\Spec k\hookrightarrow X$ be a Zariski point in $X$ with closure $Y$. The situation is said to satisfy \emph{continuity} if for any such $\epsilon_Y$:
		\begin{enumerate}[i.]
			\item (Essentially surjective) Let $a\in D(k)$. 
			Then $\Spec k$ has a neighborhood $\Spec k\xrightarrow{h}U\subset Y$  open dense in $Y$ and an object $\bar a\in D_U$ such that $h^*(\bar a) = a$.
			\item (Full) Let $a, b\in D_Y$. Then for any morphism $\alpha\in \hom(\epsilon_Y^*a,\epsilon_Y^*b)$, there is a $\Spec k\xrightarrow{g} U\xrightarrow{f} Y$ open dense, and a map $\bar\alpha\in \hom(f^*a, f^*b)$, such that $\alpha = g^*\bar\alpha$.
			\item (Faithful) Let $a, b\in D_Y$. Then for any morphism $\bar\alpha\in \hom(a,b)$, such that $\epsilon^*(\bar \alpha)=0$, there is a $\Spec k\xrightarrow{g} U\xrightarrow{f} Y$ open dense, $f^*(\bar\alpha)=0$.
		\end{enumerate}
	\end{definition}
	\begin{definition}[Punctual gluing]\label{gluing:spreadingOut}
		Assume that for each $\Spec k\hookrightarrow X$, we are given a $t$-structure (respectively, a weight structure) $(D^{\le}(k), D^{>}(k))$ on the category $D(k)$. For any $U\in Sub(X)$, define
		\begin{align*}
			D^{\le}(U)	:=\{a&\in D_U\big| \epsilon^*(a)\in D^{\le}(k)\text{ for }\epsilon:\Spec k\rightarrow U\text{ any point of }U\} \\
			D^{>}(U)	:=\{a&\in D_U\big| \epsilon^!(a)\in D^{>}(k)\text{ for }\epsilon:\Spec k\rightarrow U\text{ any point of }U\}
		\end{align*}
		as full subcategories. In particular if $f:S\hookrightarrow T$ is an immersion, $S, T\in Sub(X)$:
		\begin{align*}
			f^*(D^{\le}(T))\subset D^{\le}(S)& &f^!(D^{>}(T))\subset D^{>}(S)
		\end{align*}
	\end{definition}
	\begin{definition}[Continuity for $t$-structures and weight structures]\label{gluing:continuityForT}
		Assume the situation of \ref{gluing:spreadingOut} with $(D^{\le}(k), D^{>}(k))$ a $t$-structure (respectively a weight structure) on $D(k)$. Let $\epsilon_Y:\Spec k\hookrightarrow X$ be a Zariski point in $X$ with closure $Y$. 
%
		The situation is said to satisfy \emph{continuity of $t$-structures} (respectively \emph{continuity of weight structures}) if for any such $\epsilon_Y$:
		\begin{enumerate}[i.]
			\item (Continuity for the negative part) For any $a\in D^{\le}(k)$, there is a neighborhood $U$ of $\Spec k$ in $Y$, $\Spec k\xrightarrow{f}U\subset Y$ and an object $\bar a\in D^{\le}(U)$ with $f^*\bar a = a$.
			\item (Continuity for the positive part) For any $a\in D^{>}(k)$, there is a neighborhood $U$ of $\Spec k$ in $Y$, $\Spec k\xrightarrow{f}U\subset Y$ and an object $\bar a\in D^{>}(U)$ with $f^!\bar a = a$.
		\end{enumerate}
	\end{definition}

	We are now in a position to prove our main result:
	\begin{theorem}[Punctual gluing]\label{gluing:mainresult}
		Assume continuity \ref{gluing:continuity} and continuity of $t$-structure (respectively, weight structure) \ref{gluing:continuityForT} on a Noetherian scheme $X$. Then $(D^{\le}(X), D^{>}(X))$ is a $t$-structure (respectively, a weight structure) on $D_X$.
	\end{theorem}		
	\begin{proof}
			(Karoubi-closed) Since both $f^*$ and $f^!$ preserve direct sums, the claim is immediate since $D^{\le}(k)$ resp. $D^>(k)$ are Karoubi-closed by assumption. 	
	
			(Orthogonality) Let $a\in D^{\le}(X)$ and $b\in D^{>}(X)$. We do a Noetherian induction on $X$. The base case is when $X=\Spec k$, in which case it follows since $D^{\le}(k)\perp D^{>}(k)$. Let $\eta:\Spec K\rightarrow X$ be a generic point of $X$. Then $\hom(\eta^*(a), \eta^*(b)) = 0$ by definition, using  $D^{\le}(K)\perp D^{>}(K)$ (notice that $\eta^!:=\eta^*$ in this case, by definition \ref{gluing:extended4functors}). Let $f\in \hom(a,b)$, then since $\eta^*(f)=0$, there is a a $j:U\subset X$ open, such that $j^*(f)=0$. Since $j^*j_* = \id$, it follows that the composite: $a\xrightarrow{f} b \rightarrow j_*j^*b$ vanishes. Hence $f$ factors through the third term in the triangle, $i_*i^!b$, where $i:Z=X-U\hookrightarrow X$ is the complement. But $\hom(a,i_*i^!b)=\hom(i^*a, i^!b)=0$ by induction hypothesis and we are done.

			(Invariance/Co-invariance) Invariance or co-invariance is immediate since $f^*$, $f^!$ commute with shifts.
			
			(Decomposition) We do a Noetherian induction on $X$. The base case is when $X=\Spec k$ in which case it follows since $(D^{\le}(k),D^{>}(k))$ decompose $D(k)$.
			
			Let $h\in D_X$ and $\eta:\Spec K\rightarrow X$ be a generic point of $X$. Then, there is a triangle $a \rightarrow \eta^*h\rightarrow b \rightarrow$ with $a\in D^{\le}(K)$, and $b\in D^{>}(K)$. By continuity of $t$-structure, and restricting neighborhoods if necessary, there is a $j:U\hookrightarrow X$ open, neighborhood of $\eta$, and $\bar a\in D^{\le}(U)$, $\bar b\in D^{>}(U)$, with $a=\eta^*\bar a, b=\eta^*\bar b$. Therefore, by continuity, restricting $U$ even further if needed, we get morphisms $\bar a\rightarrow j^*h\rightarrow \bar b\rightarrow$ composing to $0$, and therefore also inducing maps $Cone(\bar a\rightarrow j^*h)\rightarrow \bar b$. This map is an isomorphism, possibly by restricting U, and hence we can even assume that $\bar a\rightarrow j^*h\rightarrow \bar b\rightarrow $ is distinguished.

			Let $i:Z\rightarrow X$ denote the complement. Consider the triangles:
			\begin{align*}
				 \text{(defining }u\text{): } & u\xrightarrow{\alpha} h\rightarrow j_*\bar b\rightarrow & \text{(t1)}\\
				 \text{(defining }v\text{): } & v\xrightarrow{\beta} u\rightarrow i_*\tau^Z_{>}i^*u\rightarrow \text{ with } \tau^Z_{>} \text{ being the truncation to }D^{>}(Z) &\text{(t2)}\\
				 \text{(defining }w\text{): } & v\xrightarrow{\alpha\beta}h\rightarrow w\rightarrow & \text{(t3)}\\
				 \text{(via octahedron axiom): } & i_*\tau^Z_{>}i^*u\rightarrow w\rightarrow j_*\bar b\rightarrow & \text{(t4)}
			\end{align*}
			where $\tau^Z_{>}$ is well defined by Noetherian induction. Then we have:
			\begin{align*}
				i^*v&\cong \tau^Z_{\le}i^*u \in D^{\le}(Z)	& &\text{ using }i^*\text{ applied to (t2)}\\
				j^*v&\cong j^*u\cong \bar a \in D^{\le}(U) & &\text{ using }j^*\text{ applied to (t2), (t1)}\\
				i^!w&\cong \tau^Z_{>}i^*u \in D^{>}(Z) 	& &\text{ using }i^!\text{ applied to (t4)}\\
				j^*w&\cong \bar b \in D^{>}(U)			& &\text{ using }j^*\text{ applied to (t4)}
			\end{align*}
			
		Since each Zariski point of $X$, lies in either $Z$ or $U$, it follows that $v\in D^{\le}(X)$ and $w\in D^{>}(X)$. Then the triangle (t3) implies decomposition of an arbitrary $h$, as required.
	\end{proof}

We prove the following result for completeness, it shows that several natural constructions correspond to a glued $t$-structure:
\begin{proposition}\label{gluing:unglue}
	Fix a scheme $X$ and assume we have extended four functors \ref{gluing:extended4functors} and continuity \ref{gluing:continuity}. Assume that for each $Y\in Sub(X)$ we have a $t$-structure (resp. a weight structure) $(D^{\le}(Y), D^{>}(Y))$ on $D_Y$ such that for any immersion $f:Y\hookrightarrow Y'$,
	\begin{align*}
		f^*(D^{\le}(Y')\subset D^{\le}(Y) & &		f^!(D^{>}(Y')\subset D^{>}(Y)
	\end{align*}
	For each $\epsilon:\Spec K\hookrightarrow X$ define:
	\begin{align*}
		D^{\le}(K) = \bigcup_Y\epsilon_Y^*(D^{\le}(Y))	& &	D^{>}(K) = \bigcup_Y\epsilon_Y^!(D^{>}(Y))
	\end{align*}
	where $\epsilon_Y:\Spec K\hookrightarrow Y\rightarrow X$ denotes any factorization of $\epsilon$ as $Y$ varies over subschemes of $X$ containing $\Spec K$.
	
	Then $(D^{\le}(K), D^{>}(K))$ forms a $t$-structure (resp. a weight structure) on $D(K)$ satisfying continuity of $t$-structures, and the $t$-structure $(D^{\le }(Y), D^{>}(Y))$ on $D(Y)$ is obtained by gluing \ref{gluing:mainresult}.
\end{proposition}
\begin{proof}
	\emph{$(D^{\le }(K), D^{>}(K))$ is a $t$-structure} (resp. a weight structure): 
	First note that if $Z$ is closure of $\Spec K$ in $X$, then $f^*, f^!:D^{\le }(Y)\rightarrow D^{\le }(Y\cap Z)$ (resp. $f^!:D^{> }(Y)\rightarrow D^{> }(Y\cap Z)$ is essentially surjective -- if $a\in D^{\le }(Y\cap Z)$ (resp. $a\in D^{>}(Y\cap Z)$) then if we let $b=f_!a$ (resp. $b=f_*a$), we have that $b\perp D^{>}(Y)$ (resp. $D^{\le Y}\perp B$) by adjunction and $f^*b=f^!b=b$.
	
	Therefore we can replace $X$ by $Z$ and assume that $\Spec K$ is generic in $X$ which is irreducible. $\epsilon_Y^*=\epsilon_Y^!$ since $Y$ must be open. Now we get:
	\begin{itemize}
		\item (Karaubi-closed) Clear, since $f^*, f^!$ preserve direct sums and $D^{\le}(Y)$, $D^{>}(Y)$ are Karaubi-closed.
		\item (Orthogonality) $a\in D^{\le}(K)$, $b\in D^{>}(K)$. Then $a=\epsilon_Y^*\bar a$ with $\bar a\in D^{\le}(Y)$, $b=\epsilon_{Y'}^*\bar b$ with $\bar b\in D^{>}(Y')$. Restricting to $U=Y\cap Y'$ with immersions $j:U\hookrightarrow Y$ and $j':U\hookrightarrow Y'$ we have $j^*\bar a\in D^{\le}(U)$ and ${j'}^*\bar a\in DM^{>}U$. Hence $\hom(j^*\bar a, j^*\bar b)=0$, hence $\hom(a, b)=0$ by continuity.
		\item (Invariance) Clear, since $f^*, f^!$ preserve shifts.
		\item (Decomposition) Let $h\in D(K)$, then $h=\epsilon_U^*\bar h$ for $\bar h$ in $D_U$ for some $U\subset X$ open. Then we have a decomposition $a\rightarrow h\rightarrow b\rightarrow$ with $\bar a\in D^{\le}(U)$, $\bar b\in D^{>}(U)$. Therefore it follows that $\epsilon_U^*\bar a\rightarrow \epsilon_U^*\bar h\rightarrow \epsilon_U^*\bar b\rightarrow$ is a triangle. But then first (resp. last) term lies in $D^{\le}(K)$ (resp. $D^{>}(K)$) by definition.
	\end{itemize}
	
	Finally to see that the given $t$-structure (resp. a weight structure) is a glued one, we first notice that continuity for $t$-structures (resp. weight structures) \ref{gluing:continuityForT} is obvious by the way subcategories are define. Then, gluing gives rise to a $t$-structure, let us call it $(D^{\le t_1}(X), D^{> t_1}(X))$. Then, by definition, $D^{\le}(X)\subset D^{\le t_1}(X)$ and $D^{>}(X)\subset D^{> t_1}(X)$. But then $D^{\le}(X)^\perp \supset D^{\le t_1}(X)^\perp$ and ${^\perp D^{>}(X)}\supset {^\perp D^{>t_1}(X)}$ and the claim now follows by \ref{otherIsPerp}.
\end{proof}	

\begin{remark}
	Using this it is easy to show that the perverse $t$-structure of \cite{BBD} as well as the Morel's $t$-structures as constructed in \cite{arvindVaish} (which are a mild generalization of the original construction in \cite{morelThesis}) are obtained by punctual gluing. It can also be easily seen that \ref{gluing:mainresult} can be used to give a different construction of the $t$-structures in these two cases.
\end{remark}

\section{Review of motivic sheaves}\label{sec:motives}
		In this section we review the (triangulated category) of motivic sheaves with rational coefficients, where our main applications live. We summarize the limited properties of the motivic sheaves we use below:
\begin{para}\label{motives:6functors}\label{motives:comparision}\label{motives:vanishingsForSmooth}
	Associated to any separated scheme $X$ of finite type over any field $k$ one can associate a category $DM(X)$ such that following properties hold:
	\begin{enumerate}
		 \item $DM(X)$ is a $\Q$-linear rigid tensor triangulated category which is idempotent complete, and for any morphism $f$ from $Y\rightarrow X$ we are given adjoint morphisms:
			 	\begin{align*}
					f^*: DM(X) \leftrightarrows DM(Y) : f_* & \text{ for }f\text{ arbitrary}\\
					f_!: DM(Y) \leftrightarrows DM(X) : f^! & \text{ for }f\text{ separated of finite type over }k
				\end{align*}
				with $f^*$ monoidal, and isomorphism of functors
				\begin{align*}
					(f g)_* \iso f_*g_* & & (f g)_! \iso f_!g_! & & (f g)^* \iso g^*f^* & & (f g)^! \iso g^!f^!
				\end{align*}
				whenever both sides make sense. We let $1_X$ denote the unit for the monoidal structure in $DM(X)$. In each $DM(X)$ there is an invertible (for the monoidal structure) object $1_X(1)$ (Tate motive) which is preserved under the four functors. We let $M(r):=M\otimes 1_X(1)^{\otimes r}$ for any $M$ in $DM(X)$ and any $r\in \Z$.
			\item We have a natural transformation:
				\begin{align*}
					f_!	\rightarrow f_*
				\end{align*}
				which is an isomorphism for $f$ proper.
			\item (Localization) If $i:Z\hookrightarrow X$ is a closed immersion and $j:U\hookrightarrow X$ denote the open complement, then the functors $i^*,i_*,i^!,j_!,j^*,j_*$ satisfy the formalism of gluing \ref{gluing:openClosed}.
			\item (Base change) Any cartesian square (for morphisms of finite type) leads to natural isomorphisms as follows:
				\[\begin{array}{lcl}
					\begin{tikzcd}
						Y'\ar[r, "f'"]\ar[d, "g'"]	&X'\ar[d, "g"]		\\
						Y\ar[r, "f"]				&X	
					\end{tikzcd}
							& \Rightarrow 
							 \begin{array}{l}
								f^*g_! \overset\iso\longrightarrow g'_!{f'}^{*} \\
								g'_*{f'}^{!} \overset\iso\longrightarrow f^!g_{*}
							  \end{array}
				\end{array}\]
				For $g$ proper, identifying $g_!$ with $g_*$ yields so called \emph{proper base change}, and in this case the isomorphism holds without finite type assumption. For $f$ smooth, purity below allows us to replace $f^!$ with $f^*$ and yields so called \emph{smooth base change}.
			\item (Purity)
					If $f:Y\rightarrow X$ is quasi-projective, smooth, of relative dimension $d$ then for any object $M\in DM(X)$, we have functorial isomorphisms:
				\[
					f^!(M)(-d)[-2d]\iso f^*(M) \text{ where } d=\dim Y-\dim X
				\]
				 (Absolute purity) If $i:Z\hookrightarrow X$ denote a closed immersion of everywhere co-dimension $c$ with $Z,X$ regular then we have an isomorphism:
				\[
					i^! 1_X \iso  1_Z(-c) [-2c]
				\]
				where $1_Y$ is the unit of the monoidal structure in $DM(Y)$.
		\item (Separated-ness) If $f:Y \rightarrow X$ is a finite surjective universal homeomorphism then $f^*$ induces an equivalence $DM(X)\rightarrow DM(Y)$ and hence $(f^*\iso f^!, f_*\iso f_!)$ form inverses up to equivalence.
%
		\item (Comparison with (motivic) cohomology)  For any $\pi:X\rightarrow \Spec k$ with $X$ regular and connected and $k$ perfect. Then we have:
		\[
			\Hom_{DM(\Spec k)}( 1_k,\pi_* 1_X(q)[p]) 	=  H^{p,q}(X) = \begin{cases}
																						0 & p>2q\text{ or }p>q+\dim X\\
																						0 & q\le1\text{ and }(p,q)\notin\{(0,0),(1,1),(2,1)\}\\
																						\mathbb{Z}(X)\otimes\Q & (p,q)=(0,0)\\
																						\mathcal{O}^{*}(X)\otimes\Q & (p,q)=(1,1)\\
																						Pic(X)\otimes\Q & (p,q)=(2,1)
																		\end{cases}
		\]
		where $H^{p,q}(X)$ denotes the motivic cohomology of $X$ with rational coefficients.
		\item (Continuity) Assume $\{X_i\}_{i\in I}$ is a pro-object in category of schemes where $I$ is a small co-filtering category, such that the transition maps $g_{j\rightarrow i}:X_j\rightarrow X_i$ are affine for all maps $j\rightarrow i$ in $I$. Let $X=\lim _{i\in I}X_i$ in the category of schemes (which exists because of the assumption on $g_{j\rightarrow i}$) and let $f_i:X\rightarrow X_i$ denote the natural map. Let $i_0$ be any object in $I$, and let $I/i_0$ be the category over $i_0$: 
		\begin{align*}
			 \forall A\in DM(X) &\;\;\;\;\; \exists j\in I, j\in A' \text{ s.t. }A\iso f_j^*(A')& &\\
			\forall A, B\in DM(X_{i_0}) &\;\;\;\; \lim_{\underset{j\in I/i_0}\longleftarrow}\Hom_{DM(X_i)}(g^*_{j\rightarrow i_0}A,g^*_{j\rightarrow i_0}B) = \Hom_{DM(X)}(f_{i_0}^*A,f_{i_0}^*B)
		\end{align*}
		\end{enumerate}
\end{para}
\begin{para}\label{motives:DAandDM}
	One choice for such a construction is the category of motivic sheaves without transfers as constructed by Ayoub in \cite{ayoub_thesis_1, ayoub_thesis_2}. This is the category $\mathbb{SH}_\mathfrak{M}^T(X)$ of \cite[4.5.21]{ayoub_thesis_2} with $\mathfrak{M}$ being the complex of $\Q$-vector spaces (and one works with the topology \'etale topology), also denoted as $DA(X)$ in the discussion \cite[2.1]{ayoub2012relative}. The second choice for the construction is the compact objects in the category of motivic sheaves with transfers, the Beilinson motives $DM_{B,c}(X)$ as described in the article \cite{cisinski2012triangulated}. 
\end{para}
\begin{remark}\label{remark:qp}
		Constructions of \cite{ayoub_thesis_1, ayoub_thesis_2} require that all schemes are assumed to be quasi-projective (see \cite[\S 1.3.5]{ayoub_thesis_1}) over the base. Since such objects are stable under immersions, base change, etc., and also because Jong's alterations can be chosen to be quasi-projective \cite{jongAlterations}, this does not cause any problem in the proofs with this additional restriction.
\end{remark}
Our assumptions are then valid in both these situations: 
\begin{proposition} 
	The assumption \ref{motives:6functors} holds if we work either with $DA(X)$ (and restrict to quasi-projective $X$) or $DM_{B,c}(X)$ of \ref{motives:DAandDM}.
\end{proposition}
\begin{proof}
	For $DA(X)$ \cite[4.5]{ayoub_thesis_2} shows that $DA(-)$ gives rise to an homotopy-stable algebraic derivator in the sense of \cite[2.4.13]{ayoub_thesis_1}. The axiom DerAlg 5 allows one to invoke the results of \cite[1.4]{ayoub_thesis_1}. In particular, \cite[1.4.2]{ayoub_thesis_1}  gives the existence of the four functors $f^*, f^!, f_*, f_!$ satisfying $(1), (2), (4)$. Purity $(5)$ follows from \cite[1.6.19]{ayoub_thesis_1} (also see \cite[2.14]{ayoubrealisation}) , while absolute purity is shown in \cite[Cor. 7.5]{ayoubrealisation}. Localization $(3)$ is \cite[4.5.47]{ayoub_thesis_2}. For separatedness $(6)$ by adjointness it is enough to show that $f^*$ is an equivalence which follows from \cite[3.9]{ayoubrealisation}. Property $(8)$ is shown in \cite[Prop. 2.30]{ayoubrealisation}. 
	Property $(7)$ follow from the fact that there are equivalences $a^{tr}:DA(k)\cong DM(k):o^{tr}$, for $k$ perfect where the right hand side is the corresponding construction with transfers due to \cite{voevodsky} (see \cite[2.4]{ayoub2012relative}) for which we know the computations of motivic cohomology by \cite[3.5, 4.2, 19.3]{weibel_mazza_notes} (or see \cite[2.2.6]{vaish2016motivic} for details).

	For the Beilinson motives $DM_{B,c}(X)$ all the properties of \ref{motives:6functors} except $(6), (7)$ are already discussed in the introduction of \cite{cisinski2012triangulated}, section $C$ -- properties $(1)$ to $(5)$ and $(8)$ are part of Grothendieck's six functor formalism as defined in \cite[A.5.1]  {cisinski2012triangulated} (continuity is an extension of usual six functor formalism as mentioned in their list) and hold for $DM_{B,c}(X)$ by \cite[C.Theorem 3]{cisinski2012triangulated}. Separatedness $(6)$ can be deduced from \cite[Prop. 2.1.9]{cisinski2012triangulated}. Property $(7)$ can be deduced from the corresponding computation in Voevodsky's category, see \cite[3.5, 4.2, 19.3]{weibel_mazza_notes} (this is also present in \cite[11.2.3]{cisinski2012triangulated}).
\end{proof}
%
\begin{para}Our constructions will happen in the category of cohomological motives -- these are the cohomological analogues of (that is, dual of) effective motives. The key properties of the category of cohomological motives are 
\begin{itemize}
	\item It is (finitely) generated by motives of the form $\pi_*1_X$ for $\pi$ (essentially) smooth, projective.
	\item It includes all motives of the form $\pi_*1_X$ for $\pi$ arbitrary.
	\item It is stable under $f^*, f_*, f^!, f_!$ for $f$ immersions, and $f^*$ for $f$ arbitrary, hence satisfying the formalism of gluing and extended gluing.
\end{itemize}
These properties are standard and have been well covered in literature, see e.g. \cite[\S 3.1]{ayoub2012relative} or \cite[\S 3]{vaish2016motivic}. We briefly summarize the situation below. We throughout assume the situation of \ref{motives:6functors}.\end{para}

%
%
\begin{definition}\label{definieCohMotives}
	Let $Y$ be a Noetherian scheme of finite type over $k$.  We define the following collection of objects of $DM(Y)$ (also identified with the corresponding full subcategory):
	\begin{align*}
		S^{coh}(X)	= \{ p_*1_X \big| p:Y\rightarrow X\text{ projective}, Y\text{ regular}\}
		& &DM^{coh}(X) = \vvspan{S^{coh}(X)}
	\end{align*}
	$DM^{coh}(X)$ is called as the \emph{category of cohomological motives} over $X$. 
\end{definition}
We can replace projectivity in the definition of $S^{coh}(X)$ by properness:
\begin{lemma}[Motivic Chow's lemma]\label{motivicChowLemma}
	Let $p:Y\rightarrow X$ be proper, with $Y$ regular. Then there is a map $q:Y'\rightarrow Y$ with $p\circ q:Y'\rightarrow X$ projective, $\dim Y'=\dim Y$, $Y'$ regular such that $p_*1_Y$ is a retract of $(p\circ q)_*1_{Y'}$.
\end{lemma}
\begin{proof}
	This is \cite[3.1]{hebert2011structure}.
\end{proof}
	It is possible to work with $f$ arbitrary for definition of $S^{coh}(X)$:
	\begin{lemma}\label{arbitraryY}
		For any morphism $f:Y\rightarrow X$, we have that $f_* 1_Y(-c)\in DM^{coh}(X)$, where $c\ge 0$. 
	\end{lemma}
	\begin{proof}
	This is \cite[3.1.7]{vaish2016motivic}, where we use \ref{motivicChowLemma} to reduce to our situation.
	\end{proof}


This allows us to prove stability properties of $DM^{coh}(X)$ which are relevant for formalism of gluing:
	\begin{proposition}[Stability of cohomological motives]\label{t-Structure:stabilityOfCohomologicalMotives}
		Let $f:Y\rightarrow X$ be a morphism of schemes of finite type. Then:
		\begin{enumerate}[i.]
			\item 
				For $f$ arbitrary, $f_*(DM^{coh}(Y))\subset DM^{coh}(X)$.
			\item 
				For $f$ arbitrary, $f^*( DM^{coh}(X)) \subset  DM^{coh}(Y)$
			\item 
				For $f$ quasi-finite (e.g. an immersion), $f_!( DM^{coh}(Y)) \subset  DM^{coh}(X)$.
			\item 
				For $f$ an immersion $f^!( DM^{coh}(X)) \subset  DM^{coh}(Y)$.
		\end{enumerate}
		Additionally, if $\epsilon:\Spec K\rightarrow X$ is any Zariski point, then $\epsilon^*(DM^{coh}(X))\subset DM^{coh}(K)$.
	\end{proposition}
	\begin{proof}
		This is \cite[3.1.9]{vaish2016motivic} except for the additional claim. Additional claim can be verified on generators and proper base change tells us that $\epsilon^*p_*1_Y = p'_*1_{Y'}$ where $p':Y'\rightarrow \Spec K'$ is pullback of $p:Y\rightarrow X$ projective, and hence lies in $DM^{coh}(K)$ using \ref{arbitraryY}.
	\end{proof}

\begin{proposition}[Formalism of gluing, continuity]\label{coh:continuity}
	The extended formalism of gluing \ref{gluing:extended4functors} and continuity \ref{gluing:continuity} for any $X$, for 
			$D_Y=DM^{coh}(Y)$ and $D(K)=DM^{coh}(K)\cong DM^{coh}(K^{perf})$

%
%
\end{proposition}
\begin{proof}
		Note that $DM^{coh}(K)\xrightarrow{\cong}DM^{coh}(K^{perf})$ by using separatedness and continuity. The extended formalism of gluing holds for this setup since the same holds for the full category $DM(-)$, and \ref{t-Structure:stabilityOfCohomologicalMotives} tells us that the functors restrict to $DM^{coh}(-)$.
\end{proof}

\section{Examples and applications}\label{sec:applications}

\makeatletter
\renewcommand{\themaster}{\thesubsection.\arabic{master}}
\@addtoreset{master}{subsection}
\makeatother

Below we give several examples and applications for punctual gluing. Note that the results in \ref{sec:application:bondarko} and \ref{sec:application:artinpicard} are not new and have independently appeared, but the main result in \ref{sec:application:1motivic} is new.

\subsection{Bondarko's weight structure}\label{sec:application:bondarko} As a first example we demonstrate how to recover the Chow weight structure on $DM(X)$ as constructed in \cite{hebert2011structure} and independently in \cite{bondarko2014weights} beginning with the Bondarko's construction of the Chow weight structure on $DM(k)$ where $k$ is perfect.

	Recall that in \cite[\S 6.5]{bondarko_weights}, Bondarko constructs a bounded weight structure on $DM_{gm}(k)$  ($=DM(k)$ in our notation) satisfying the following properties:
	\begin{enumerate}
		\item The heart of the weight structure, $H(k):=DM_{w\le 0}(k)\cap DM_{w> 0}(k)[-1]$, is given by:
			\[
				H(k) = \mathcal R (p_*1_X(c)[2c]\mid p:X\rightarrow \Spec k\text{ smooth, proper}, c\in \Z)
			\]
		\item We have an explicit description of the positive and the negative part of the weight structure as:
			\begin{align*}
				DM_{w\le 0}(k) = \mathcal R(Ext(\cup_{n\le 0} H(k)[n] ))& &DM_{w> 0}(k) = \mathcal R(Ext(\cup_{n> 0} H(k)[n] ))
			\end{align*}
	\end{enumerate}
	The first property follows by \cite[\S 6.6]{bondarko_weights}, and the second can be deduced by noting that $H\oplus H\subset H$ and $\mathcal R(H) = H$ using the calculation \cite[1.10]{hebert2011structure}.
	
	Then the following recovers the key result of \cite{hebert2011structure} and \cite{bondarko2014weights}:
	\begin{theorem}\label{application:weights}
		The extended formalism of gluing \ref{gluing:extended4functors} and continuity \ref{gluing:continuity} for any $S$, for $D_Y=DM(Y)$ and $D(K)=DM(K)\cong DM(K^{perf})$. The weight structure $(DM_{w\le 0}(K^{perf}), DM_{w> 0}(K^{perf}))$ on $D(K)=D(K^{perf})$ satisfies continuity of weight structures \ref{gluing:continuityForT}, and hence gives rise to a glued weight structure on $DM(S)$. 
	\end{theorem}
	\begin{proof}
		Formalism of gluing and continuity hold by \ref{motives:6functors}. 
		
		Let $a, a'\in D(K)$, with $\delta:\Spec K \rightarrow U\subset X$ with $a = \delta^*\bar , a'=\delta^*\bar a'$ for some $\bar a, \bar a'\in D_U$. Then, by continuity, if $b\in \mathcal R(\{a\})$ (resp. $b\in Ext^1(\{a\},\{a'\})$), by restricting $U$ to a smaller neighborhood of $\Spec K$ if necessary, there is a $\bar b\in R(\{\bar a\})$ (resp. $\bar b\in Ext^1(\{\bar a\},\{\bar a'e\})$) such that $\delta^*\bar b=b$.
		
		Further, since $\delta^*, \delta^!$ commutes with shifts, it follows that it is enough to show that for any $a\in H(K^{perf})$, there is a neighborhood $U$ of $\Spec K$, $\delta:\Spec K\rightarrow U$, and a $\bar a\in D_U$ with $\delta^*\bar a = a$ such that for any $\epsilon: \Spec L\rightarrow U$, $\epsilon^*\bar a \in H(L^{perf})$. In fact, we can even assume $a = p_*1_X(c)[2c]$ with $p:X\rightarrow \Spec K^{perf}$ proper, smooth. Since $\delta^*$ and $\delta^!$ commute with shifts and twists and $H(k)(c)[2c] = H(k)$ for all $c$, we can even assume $c=0$. Then the claim follows from the next lemma.
	\end{proof}
\begin{lemma}\label{lemma:boat}
	Let $\epsilon:\Spec K\rightarrow S$ be any point and let $a=p_*1_X\in D(K) = DM(\Spec K^{perf})$ with $p:W\rightarrow \Spec K^{perf}$ proper, smooth. Then we can find a neighborhood $U$ of $\Spec K$, $\epsilon:\Spec K\overset{\epsilon_U}\hookrightarrow S$, and an element $\bar a \in D_U$ with $\epsilon^*\bar a = \epsilon^!\bar a = a$ and such that if $\delta:\Spec L \rightarrow U$ is any point of $U$, 
		\begin{align*}
			\delta^*\bar a = p_{L*}1_{W_L}& &\delta^!\bar a = p_{L*}1_{W_L}(-c)[-2c]
		\end{align*}
		with $p_L:W_L\rightarrow \Spec L^{perf}$ proper smooth and $c>0$ if $\delta\ne \epsilon$.
\end{lemma}
\begin{proof}
		Given $p:W\rightarrow \Spec K^{perf}$ smooth proper, $p$ is defined over some $K'/K$ finite purely inseparable. That is there is a smooth proper $f$ as in the diagram below, such that $W$ is a pullback of $X$ along $t$. Then we can spread out $f$, $s$, to $\bar f$, $\bar s$ on some irreducible neighborhood $U$ of $\Spec K$ such that $\bar f$ is proper smooth and $\bar s$ is finite radicial as shown in the diagram below.

		We have $\delta^!\bar a =\underline\delta^*i^!\bar a$, where $ \delta = i\circ \underline\delta$ is a factorization of $\delta$ with $\underline\delta$ generic in $Y\subset U$. By restricting $U$ if necessary, we can assume $Y$ as well as $U$ are regular. 
Then by restricting $U$, we also assume that $\bar X, U'$ are regular. Now we can construct the  diagram below as follows: $s_Y, s'$ are pullback of $\bar s$; $Y'_{red}, V'_{red}$ are the reduced subvariety underlying $Y', V'$; $\delta'$ which is a pullback of $\underline\delta$, gives rise to the map $v$. Since $\bar s$ is finite radicial, so is $s\circ s_1$ and hence the perfection $r_1$ factors via $t_1$. $f_Y, f_1, f'_1$ are pullback of $\bar f$.
	\[\hspace{-0.1\linewidth}\begin{tikzcd}
		W\ar[d, "\circ" description, "p"]\ar[rd]& & & & & &X'_1\ar[d, "\circ" description, "f'_1"]
		 \\
		\Spec K^{perf}\ar[rd, "t"]\ar[rdd, "r" below]& X\ar[r]\ar[d, "\circ" description, "f"] &\bar X\ar[d, "\circ" description, "\bar f"]\ar[rr, hookleftarrow, "l_X"]& &X_Y\ar[d, "\circ" description, "f_Y"]\ar[r, hookleftarrow, "v''"]&X_1\ar[d, "\circ" description, "f_1"]\ar[ru, leftarrow, "t''"]
										&\Spec L^{perf}\\
		& \Spec K'\ar[d, "\ast" description, "s"]\ar[r]& U'\ar[r, "i'", hookleftarrow]\ar[d, "\ast" description, "\bar s"] &Y'\ar[rrd, hookleftarrow, "\delta'"]\ar[d, "\ast" description, "s_Y"]\ar[r, "\ast" description, "l", hookleftarrow]	   &Y'_{red}\ar[r,hookleftarrow, "v"] &V'_{red}\ar[d, "\ast" description, "s_1"]\ar[ru, leftarrow, "t_1"]\\
	\begin{minipage}{0.4\linewidth}
		$-\circ\rightarrow$ smooth\\
		$-\ast\rightarrow$ finite radicial
	\end{minipage}\hspace{-0.2\linewidth}
	& \Spec K\ar[r, "\epsilon"]& U\ar[rrrd, leftarrow, "\delta" below]\ar[r, "i", hookleftarrow]	&Y\ar[rrd, hookleftarrow, "\underline\delta"]	&	   &V'\ar[d, "\ast" description, "s'" left]\\
		& &\- & & &\Spec L\ar[ruuu, leftarrow, "r_1" right]	
	\end{tikzcd}\]
	It is clear that:
	\begin{align*}
		\underline\delta^*i^*\bar s_*\bar f_*1_{\bar X} &= \underline\delta_*s_{Y*}i^{'*}\bar f_*1_{\bar X} = s'_*\delta^{'*}i^{'*}\bar f_*1_{\bar X} = s'_*s_{1*}s_{1}^*\delta^{'*}i^{'*}\bar f_*1_{\bar X} = s'_*s_{1*}v^*l^*i^{'*}\bar f_*1_{\bar X} \\
			& = s'_*s_{1*}v^*l^*i^{'*}\bar f_*1_{\bar X} = s'_*s_{1*}v^* f_{Y*}l_X^*1_{\bar X} = s'_*s_{1*}f_{1*}v^{''*}l_X^* 1_{\bar X} = s'_*s_{1*}f_{1*}1_{X_1}
	\end{align*}
	using separatedness and base change. It follows that, restricting to $DM(L^{perf})$, 
	\begin{align*}
		r_1^*\delta^* \bar a &= r_1^*\underline\delta^*i^*\bar s_*\bar f_*1_{\bar X} = r_1^*s'_*s_{1*}f_{1*}1_{X_1} = t_1^*s_1^*s^{'*}s'_*s_{1*}f_{1*}1_{X_1} \\
		&= t_1^*f_{1*}1_{X_1} = f'_{1*}t^{''*} 1_{X_1} = f'_{1*}1_{X'_1}
	\end{align*}
	and therefore we can define $\bar a=\bar s_*\bar f_*1_{\bar X}$, and for $L$, $W_L:= X'_1$ and $p_L:=f'_1$. Note that $f'_1$ is smooth and proper since $\bar f$ is and we are done. 
	
		For $\delta^!(\bar a)$, note that:
	\begin{align*}
		\underline\delta^*i^!\bar s_*\bar f_*1_{\bar X} &= \underline\delta_*s_{Y*}i^{'!}\bar f_*1_{\bar X} = s'_*\delta^{'*}i^{'!}\bar f_*1_{\bar X} = s'_*s_{1*}s_{1}^*\delta^{'*}i^{'!}\bar f_*1_{\bar X} = s'_*s_{1*}v^*l^*i^{'!}\bar f_*1_{\bar X} \\
			& = s'_*s_{1*}v^*l^!i^{'!}\bar f_*1_{\bar X} = s'_*s_{1*}v^* f_{Y*}l_X^!1_{\bar X} = s'_*s_{1*}f_{1*}v^{''*}l_X^! 1_{\bar X} = s'_*s_{1*}f_{1*}1_{X_1}(-c)[-2c]
	\end{align*}
	just as before, where we also use purity and $c=\dim X-\dim X_1=\dim U-\dim Y>0$ if $\delta\ne\epsilon$. It follows that, restricting to $DM(L^{perf})$, 
	\begin{align*}
		\delta^!\bar a = r_1^*\delta^! \bar a &= r_1^*\underline\delta^*i^!\bar s_*\bar f_*1_{\bar X} = r_1^*\underline\delta^*i^*\bar s_*\bar f_*1_{\bar X}(-c)[-2c] = f'_{1*}1_{X'_1}(-c)[-2c].
	\end{align*}
\end{proof}	

\subsection{Relative Artin and relative Picard motive}\label{sec:application:artinpicard}
	Let us consider the triangulated category of cohomological $0$-motives (Artin motives) and $1$-motives (Picard motives) -- the homological version of the same over a field were introduced and studied in \cite{ayoub20091motivic}, while the cohomological version, and over the base is studied in \cite{lehalleur2015motivic}. These are defined as:
	\begin{align*}
		DM^{coh}_{0-mot}(S)&=\vvspan{f_*1_X\big| f:X\rightarrow S\text{ proper of relative dimension}\le 0}\\	
		DM^{coh}_{1-mot}(S)&=\vvspan{f_*1_X\big| f:X\rightarrow S\text{ proper of relative dimension}\le 1}
	\end{align*}
	(In \cite{lehalleur2015motivic}, these categories are denoted as $DA^0(S), DA^1(S)$, we change the notation to avoid confusion. Also note that we allow only finite direct sums - this will be enough to recover the full construction using \ref{compactToNonCompact} and \ref{compactAreGenerated}).
	
	Below, we construct relative Artin motive ($i=0$) and relative Picard motive ($i=1$) in $DM^{coh}_{i-mot}(S)$ for any motive in $DM^{coh}(S)$. The final constructions are not really new -- for $i=0$, this is present in \cite{ayoub2012relative} and in \cite{vaish2016motivic} and for $i=1$ this appears in \cite{lehalleur2015motivic} -- but our method is more streamlined and natural.
	
	For the construction, we need to assume the same over a field, which is the result in \cite{ayoub20091motivic}. Also, while the $i=0$ case is self-contained, for $i=1$ case we depend on a certain $\Hom$ vanishing which is computed in the companion article \cite[3.3.6]{vaish2017weight}. Alternatively, for the purpose of the main application, the reader may skip the section entirely, since \ref{imotivic:overafield} is also proved in \cite[3.12]{lehalleur2015motivic} (with the relevant exactness properties computed in \cite[3.3]{lehalleur2015motivic}).
	
	We first prove the following lemma:
	\begin{lemma}\label{imotivic:overafield}
		Let $K$ be a field and $r:\Spec K^{perf}\rightarrow \Spec K$ the natural map. The equivalence $r^*:DM^{coh}(K)\leftrightarrows DM^{coh}(K^{perf}):r_*$ restricts to an equivalences of categories:
		\[
			r^*:DM^{coh}_{i-mot}(\Spec K)	\leftrightarrows DM^{coh}_{i-mot}(\Spec K^{perf})
		\]
	\end{lemma}
	\begin{proof}
		The equivalence for $DM^{coh}(-)$ is obtained by separatedness, continuity and \ref{t-Structure:stabilityOfCohomologicalMotives}, and hence we only need to show that that $r^*$ and $r_*$ preserve $i$-motivic subcategories. This can be verified on generators.
		
		For $f^*$ it is obvious by base change since $\bar X=X\otimes_K K^{perf}$ has the same dimension as $X$. For $r_*$, any such $\bar p:\bar X\rightarrow \Spec K^{perf}$ is obtained as a pullback of $p:X\rightarrow \Spec L$ with $\dim X=\dim \bar X$ and $q:\Spec L\rightarrow \Spec K$ purely inseparable field extension. Hence $f_*\bar p_* 1_{\bar X}= q_*p_*1_X$ by continuity and separatedness, and we are done.
	\end{proof}

	\begin{theorem}\label{application:picard}
		For $i\in\{0, 1\}$, the inclusion $DM^{coh}_{i-mot}(S)\hookrightarrow DM^{coh}(S)$ admits a right adjoint:
		\[
			\omega^i:DM^{coh}(S)\rightarrow DM^{coh}_{i-mot}(S)
		\]
	\end{theorem}
	\begin{proof}
		Existence of right adjoint is same as existence of $t$-structure with $DM^{coh}_{i-mot}(S)$ as the negative part, by \ref{tIsAdjoint}. 
		By \cite[2.4.1 and 2.4.7]{ayoub20091motivic} there are left adjoints
		\begin{align*}
			L\pi_0	&:DM^{eff}_c(K)\rightarrow DM_{\le 1, c}^{eff}(K)\\
			LAlb	&:DM^{eff}_c(K)\rightarrow DM_{\le 1, c}^{eff}(K)
		\end{align*}
		to the natural inclusion $DM_{\le i, c}^{eff}(K)\hookrightarrow DM^{eff}_c(K)$.
		
		Here $DM^{eff}_c(K)$ is the category of (compact objects in) effective motives of Voevodsky (that is $DM^{eff}_c(K):=DM^{eff}_{gm}(K)$), and $DM^{eff}_{\le i, c}(K)$ is finitely generated by motives of varieties of dimension $\le i$. That is:
		\begin{align*}
			DM^{eff}_c(K) &= \vvspan{p_\#1_X\big| p:X\rightarrow \Spec K\text{ projective}}\\
			DM^{eff}_{\le i,c}(K) &= \vvspan{p_\#1_X\big| p:X\rightarrow \Spec K\text{ projective}, \dim X\le i}
		\end{align*}
		 (a priori in \cite{ayoub20091motivic} one restricts to smooth schemes over $k$, but  by \cite[1.25]{lehalleur2015motivic} this is not a restriction).
		Using Verdier duality, and noting that $\mathbb D p_\#1_X = p_*1_X$, we get a right adjoint:
		\[
			\omega^i_K:DM^{coh}(K)\rightarrow DM^{coh}_{i-mot}(K)
		\]
		to the fully faithful inclusion $DM^{coh}_{i-mot}(K)\hookrightarrow DM^{coh}(K)$ (see \cite[3.12]{lehalleur2015motivic} for the detailed argument). Then we have a $t$-structure on $DM^{coh}(K)$ with 
			\begin{align*}
				D^{\le}(K) & = D^{\le}_i(K) := DM^{coh}_{i-mot}(\Spec K)\cong DM^{coh}_{i-mot}(\Spec K^{perf}) \\
				D^{>}(K) & =D^{>}_i(K):=\{a\in DM^{coh}(K^{perf})\big| D^{\le}_i(K)\perp a\}.
			\end{align*}
		
		Now formalism of gluing and continuity for $Y\mapsto DM^{coh}(Y)$, $K\mapsto DM^{coh}(K^{perf})$ is \ref{coh:continuity}. We show that the continuity for $t$-structure holds for $(D^{\le}(K), D^{>}(K))$ defined above.
		
		\emph{Continuity for $D^{\le}(K)$}: Given any generator $a=p_*1_{X}\in D^{\le}_i(K)$, $p:X\rightarrow \Spec K$ proper of dimension $\le i$, we spread it out to a morphism $\bar p:\bar X\rightarrow U$ of relative dimension $\le i$ in some neighborhood of $U$. Then $\epsilon^*\bar p_*1_{\bar X} = p_{L*}1_{\bar X_L }\in DM^{coh}_{i-mot}(L)$ for any $\epsilon:\Spec L\hookrightarrow U$, using proper base change, since $\dim (\bar X_L)\le i$ as well. 
		
		\emph{Continuity for $D^{>}(K)$}: 
		
		-- Case $i=0$ : for any $a\in {D^{>}_0}\subset D(K)$, we can find an extension $\bar a\in DM^{coh}(U)$, such that $\epsilon^!\bar a = a$ for $\epsilon:\Spec K\rightarrow U$ generic by continuity. We can also assume $U$ is irreducible. We will show that $\delta^!\bar a\in D^{>}(K)$ for some extension $\bar a$ for \emph{any} $a\in D(K)=DM^{coh}(K^{perf})$, provided $\delta\ne \epsilon$.
		
			But it is enough to assume $a = p_*1_W$ for some $p:W\rightarrow \Spec K^{perf}$ proper smooth, because such objects finitely generate $D(K)$. But then by \ref{lemma:boat} we can find $\bar a$ on some $U$ such that $\delta^!\bar a=p_{L*}1_{X_L}(-c)[-2c]$ for some $p_L:X_L\rightarrow \Spec L^{perf}$ proper smooth, $c>0$.

	 We want to show that $DM^{coh}_{0-mot}(L^{perf})\perp p_{L*}1_{X_L}(-c)[-2c]$.  It is enough to take a generator $b=g_*1_Z$ with $Z\rightarrow L^{perf}$ proper, smooth, of relative dimension $0$, in particular, one can take $Z=\Spec M$, with $g:Z\rightarrow L^{perf}$, $M/L^{perf}$ a finite separable extension. Then, since $g$ is both proper and smooth, 
	\[
		\hom (g_*1_Z, p_{L*}1_{X_L}(-c)[-2c]) = \hom (1_Z, g^*p_{L*}1_{X_L}(-2c)[-c]) = \hom (1_Z, p_{M*}1_{X_M}(-c)[-2c]) = 0
	\]
	where $p_M:X_M\rightarrow \Spec M$ is the pullback of $X$ to $\Spec M$, and this vanishes by vanishing of $H^{-c, -2c}(X_M)$ (since $c>0$).
		
		\emph{Case $i=1$}: $D^{\le }_0(K)\subset D^{\le}_1(K)$. Therefore $D^{>}_0(K)\supset D^{>}_1(K)$. By \cite[3.3.6]{vaish2017weight},
		for any $a\in D_0^>(K^{perf})$, we can find an extension $\bar a$ over some $U$, such that $\epsilon^!\bar a = a$ for $\epsilon:\Spec K\rightarrow U$ generic and $b\perp \delta^!\bar a$ for any $b=p'_*1_X[i]$ for $p':X'\rightarrow \Spec K^{perf}$ with $\dim X'\le 1$, any $i\in \Z$, and any $\delta:\Spec L\rightarrow U$, $\delta\ne \epsilon$. 
		If $A[n]\perp B$ for all $n$, then $Ext(\cup_n A[n])\perp B$, and hence $\vvspan A\perp B$. Using this for $A$ as the collection of $b$ as above, and $B=\delta^!\bar a$, we see that $D_1^{\le }(L) \perp \delta^!\bar a$ as required.

 	Thus using \ref{gluing:mainresult} and \ref{tIsAdjoint} we get the requisite adjoint.
	\end{proof}
	
\subsection{$1$-motivic $t$-structure over a base}\label{sec:application:1motivic}
	In \cite[\S 3.2]{ayoub2011nmotivic}, analogue of the still conjectural motivic $t$-structure on $DM(k)$ is constructed on a full triangulated subcategory, the subcategory of effective $1$-motives. On compact objects, this also follows from \cite{MR2102056}. This can be dualized to yield a cohomological version, a $t$-structure on the category $DM^{coh}_{1-mot}(k)$ finitely generated by motives of curves. In \cite{lehalleur2015motivic}, this construction is extended to the relative setting, and the $t$-structure is analogously constructed on the triangulated subcategory of $1$-motives over a base. 
	 
	We use \ref{gluing:mainresult} to verify that the $t$-structure of \cite{lehalleur2015motivic} in the relative setting restricts to compact objects, thereby strengthening their main result.
%
%
\begin{para}
	We want to define the $1$-motivic $t$-structure on the subcategory $DM_{1-mot}^{coh}(X)$. As a first step we need the notion of Grothendieck's six functors \ref{gluing:4functors}. Since in \ref{application:picard} $DM_{1-mot}^{coh}(X)$ is constructed as the negative part of a $t$-structure obtained by gluing, it is preserved under $f^*,f_!$. For $j$ an open immersion, we define ``$j_*$'' to be $\omega^1\circ j_*$, and for $i$ a closed immersion ``$i^!$'' to be $\omega^1\circ i^!$. To avoid confusion, we will use $i^!$ and $j_*$ only in the traditional sense in motives and not in the sense of \ref{gluing:4functors}. 	
\end{para}

	A straightforward verification now shows that this setup satisfies \ref{gluing:4functors}:

\begin{lemma}[Extended formalism of gluing]\label{1motive:formalism}
	Consider the assignment $W\mapsto D_W:=DM^{coh}_{1-mot}(W)$, and, for any open immersion $j:U\hookrightarrow W$ (resp. closed immersion $i:Z\hookrightarrow W$) the natural functors:
	\begin{align*}
		j^*: DM^{coh}_{1-mot}(W) &\rightarrow  DM^{coh}_{1-mot}(U): \omega^1\circ j_*, j_! \\ i^*, \omega^1\circ i^!: DM^{coh}_{1-mot}(W) &\rightarrow  DM^{coh}_{1-mot}(Z): i_*
	\end{align*}
	Then this setup satisfies the notion of Grothendieck's six functors \ref{gluing:4functors}. If we define 
	\begin{align*}
		D(K):=DM^{coh}_{1-mot}(K)& &\epsilon^*:DM^{coh}_{1-mot}(W)\rightarrow DM^{coh}_{1-mot}(K)
	\end{align*}
	for any point $\epsilon:\Spec K\hookrightarrow W$. Then this setup satisfies extended formalism \ref{gluing:extended4functors}.
\end{lemma}
\begin{proof}
	Let us denote the positive part of the $t$-structure on $W$ as constructed in \ref{application:picard} as $DM^>(W)$. Since this is a $t$-structure constructed by gluing, we have $\omega^1\circ f^* = f^*$, $\omega^1\circ f_!=f_!$ when applied to objects in $DM^{coh}_{1-mot}(-)$, for $f\in \{i,j\}$.
	
	The construction \ref{gluing:4functors}$(i)$ is given to us and  \ref{gluing:4functors}$(iii)$ is then $\omega^1$ applied to the same result on usual pullbacks and pushforwards. The composition law in \ref{gluing:4functors}$(ii)$ is verified below for $f_*$, case for $f^!$ is similar, and that for $f^*, f^!$ even more direct: 
	
	Let $U\overset f\hookrightarrow V\overset g\hookrightarrow W$ be any pair of immersions. Let $a\in DM^{coh}_{1-mot}(U)$ and $b\in DM^{>}(V)$ be defined by the triangle: $\omega^1(f_*a)\rightarrow f_*a\rightarrow b\rightarrow$. Applying the triangle-preserving $\omega^1 g_*$ we get 
		\[
			\omega^1 g_*\omega^1 f_*a\rightarrow \omega_1 g_*f_*a\rightarrow \omega^1g_*b\rightarrow.
		\]
		But the last term vanishes, since $g_*(DM^{>}(V))\subset DM^{>}(W)$ by construction. Hence the first map is an isomorphism, as required.
		
		Finally we verify the formalism of gluing: 
		
		If $a\in DM^{coh}_{1-mot}(W)$ and $b\in DM^{coh}_{1-mot}(U)$ we have $\hom(j^*a,b)=\hom(a,j_*b)=\hom(a,\omega^1\circ j_*b)$, the second coming from the fact that $a\in DM^{coh}_{1-mot}(W)$. This proves adjunction \ref{gluing:openClosed}$(i)$, the other cases of adjunction being similar. Vanishings \ref{gluing:openClosed}$(ii)$ follow directly from that of $j^*i_*$ and hence by adjunction for other cases. Localization \ref{gluing:openClosed}$(iii)$(resp. isomorphisms \ref{gluing:openClosed}$iv$) can be obtained by applying $\omega^1$ to the usual localization triangle (resp. isomorphism), noting that $\omega^1\circ i_*=i_*\circ \omega^1$ (resp. $\omega^1\circ j^*=j^*\circ \omega^1$).
		
		The case for extended formalism is now obvious using the continuity of the negative part of the $t$-structure in \ref{application:picard}.
\end{proof}

We now recall briefly the results from \cite{lehalleur2015motivic}:
\begin{para}[Relative $1$-motivic $t$-structure]
	Let $S$ be any scheme. Then \cite[1.1]{lehalleur2015motivic} defines:
	\begin{align*}
		DA^1(S) := \vvtspan{f_*1_X\big| f:X\rightarrow S\text{ proper of relative dimension}\le 1}\\
		DA_1(S) := \vvtspan{f_\#1_X\big| f:X\rightarrow S\text{ proper of relative dimension}\le 1}\\
		DA^1_c(S) : = \vvspan{f_*1_X\big| f:X\rightarrow S\text{ proper of relative dimension}\le 1}\\
		DA_{1,c}(S) : = \vvspan{f_\#1_X\big| f:X\rightarrow S\text{ proper of relative dimension}\le 1}
	\end{align*}
	(In particular, $DM^{coh}_{1-mot}(S) = DA^1_c(S)$. Note that we denote $\ll{-}\gg$ of \cite{lehalleur2015motivic} by $\vvtspan{-}$).
	
	One also defines (see \cite[Appendix A]{lehalleur2015motivic}):
	\begin{align*}
		\mathcal M_1(S) := \mathcal R((\text{category of Deligne }1-\text{motives over }S)\otimes \Q)
	\end{align*}
	the rational version of category of Deligne $1$-motives. Given any scheme map $f:T\rightarrow S$, there is a pullback map $f^*:\mathcal M_1(S)\rightarrow \mathcal M_1(T)$. Then, \cite[2.13 and 1.27]{lehalleur2015motivic} gives us a functor:
	\[
		\Sigma^\infty_{coh}: \mathcal M_1(S) \xrightarrow{\Sigma^\infty} DA_{1,c}(S)\xrightarrow[\cong]{(-1)}DA^1_c(S)
	\]
	such that for any morphisms of schemes $f:T\rightarrow S$, we have the compatibility \cite[2.14]{lehalleur2015motivic}:
	\[
		R_f: f^*\Sigma^{\infty}_{coh}\overset\cong\longrightarrow \Sigma^\infty_{coh}f^* : \mathcal M_1(S)\rightarrow DA^1_c(T).
	\]
	Finally, one defines a $t$ structure $t_{MM}^1(S)$ on $DA^{1}(S)$ whose heart is denoted as $\mathbf{MM}^1(S)$ (see \cite[4.9]{lehalleur2015motivic}), such that $\Sigma^\infty_{coh}(\mathcal M_1(S))\subset \mathbf{MM}^1(S)$ (see \cite[4.22]{lehalleur2015motivic}). 

	If $k$ is a field, this $t$-structure is already due to \cite{ayoub2011nmotivic} and on compact objects due to \cite{MR2102056}. We have an equivalence of $t$-categories (see \cite[4.19]{lehalleur2015motivic}, using \cite[1.27]{lehalleur2015motivic}):
	\[
		\Sigma^\infty_{coh}:(D^b(\mathcal M_1(k)), \text{standard }t\text{-structure})\xrightarrow{\cong} (DA^1_c(k), t_{MM}^1(k))
	\]
\end{para}

We now prove the main result of this section:
\begin{theorem}\label{application:t1MMrespectsCompact}
	$t^1_{MM}(S)$ restricts to a $t$-structure on $DA^1_c(S)=DM^{coh}_{1-mot}(S)$.
\end{theorem}
\begin{proof}
	We will construct a $t$-structure on $DA^1_c(S)=DM^{coh}_{1-mot}(S)$ using \ref{gluing:mainresult} and verify that it is the same as $t^1_{MM}(S)$ restricted to $DA^{1}_c(S)$. Let $(D^{\le t_{MM}^1}(S), D^{> t_{MM}^1}(S))$ denote the subcategories involved in the $t$-structure $t_{MM}^1(S)$.
	
	\underline{\emph{Construction}}: Extended formalism of gluing holds for $DA^1_c(S)$ due to \ref{1motive:formalism}. Over a field, $k$, we have the identification:
	\[
		\Sigma^\infty_{coh}:(D^b(\mathcal M_1(k)), \text{standard }t\text{-structure})\xrightarrow{\cong} (DA^1_c(k), t_{MM}^1(k))
	\]
	of \cite[4.19]{lehalleur2015motivic} (using \cite[1.27]{lehalleur2015motivic}). We let $(D^{\le}(k), D^{>}(k))$ denote the standard $t$-structure on the left and $(DA^{\le}(k), DA^{>}(k))$ the corresponding $t$-structure, the restriction of $t_{MM}^1(k)$, on the right. We will show continuity for this $t$-structure in the sense of \ref{gluing:continuityForT}.
	
	\emph{Continuity for $a\in DA^{\le}(k)$}: 
	
	 Fix $\epsilon:\Spec k\hookrightarrow S$, and let $Y$ be the closure of $\epsilon$. Let $a\in DA^{\le}(k)$ 
	 . We want to find a $\bar a\in DA^1_c(U)$ for some neighborhood $U$ of $\Spec k$ in $Y$ such that for any $\delta:\Spec l\hookrightarrow S$, $\delta^*(\bar a)\in DA^{\le}(l)$, while $\epsilon^*(\bar a)=a$. 
	 
	 $\ast$ \emph{Additional claim:} To aid in the later part of the proof, we will additionally show that $\bar a\in D^{\le t_{MM}^1}(U)$.

	 Now $a=\Sigma^{\infty}_{coh}(a')$ with $a'\in D^b(\mathcal M_1(k))$. We induct on the number of $i$ such that $H^i(a')\ne 0$.
		
	$\ast$ \emph{Base case}: $a'=x'[i]$, and $x'\in \mathcal M_1(k)$ with $i\ge 0$. Then by spreading out (see \cite[A.11]{lehalleur2015motivic}), we can find a $U\subset Y$ open dense and $\bar x'\in \mathcal M_1(U)$, such that the restriction $\epsilon^*(\bar x')=x'$. Let $\bar a = \Sigma^\infty_{coh}(\bar x'[i])$. Using the commutation isomorphism $\mathcal R_f:\Sigma^\infty_{coh}f^*\cong f^*\Sigma^\infty_{coh}$ (see  \cite[2.14]{lehalleur2015motivic}) we see that the restriction $f^*$ on $1$-motives corresponds to the pullback $f^*$ on $DA^{coh}_{1-mot}(-)$. 

	Therefore:
		\[
			\delta^*\bar a = \delta^*\Sigma^\infty_{coh}(\bar x'[i]) = \Sigma^\infty_{coh}(\delta^*(\bar x') [i])\in \Sigma^{\infty}_{coh}(\mathcal M^1(l)[i])\subset DA^{\le}(l)
		\]
		since $i\ge 0$. Also $\epsilon^*\bar a = \Sigma^\infty_{coh}(\epsilon^*\bar x'[i])= \Sigma^\infty_{coh}(x'[i]) = a$ and we are done.
	The additional claim is also immediate because $\Sigma^\infty_{coh}(\mathcal M_1(U))\subset \mathbf{MM}^1(U)$ (see \cite[4.22]{lehalleur2015motivic}) which is the heart of $t^1_{MM}(U)$.
	
	$\ast$ \emph{Induction step}: let $i$ be the smallest integer such that $H^{-i}(a')\ne 0$. Then there is a triangle:
	\[
		b':=\tau_{\le -i-1}a'\longrightarrow a' \longrightarrow c':=H^{-i}(a')\overset{h'}\longrightarrow
	\]
	with $b',c'\in D^{\le }(k)$ have fewer non-vanishing $H^i$s. By induction hypothesis, we can find a $\bar b'$, and $\bar a'$ on some $U$ satisfying requirements of continuity. By continuity of $DM^{coh}_{1-mot}(k)$, by restricting $U$ if needed, we can find an extension:
	\[
		\Sigma^\infty_{coh}(\bar b')\longrightarrow \bar a \longrightarrow \Sigma^\infty_{coh}(\bar c')\overset{\bar h}\longrightarrow 
	\]
	such that $\epsilon^*(\bar h) = \Sigma^{\infty}_{coh}(h')$. The additional hypothesis on the first and last term is that they are in $D^{\le t_{MM}^1}(U)$, and hence the same is true about $\bar a$. Now, 
	\[
		\epsilon^*(\bar a) = CoCone(\epsilon^*\bar h) \cong CoCone(\Sigma^{\infty}_{coh}(h')) \cong \Sigma^{\infty}_{coh}(a') = a
	\]
	while we have a triangle:
	\[
		\delta^*(\Sigma^\infty_{coh}(\bar b'))\longrightarrow \delta^*\bar a \longrightarrow \delta^*(\Sigma^\infty_{coh}(\bar c')) \overset{\bar h}{\longrightarrow}
	\]
	for any $\delta$. The first and the last term lie in $DA^{\le}(l)$ by induction hence so does the middle term, and we are done.
		
	\emph{Continuity of $a\in DA^{>}(k)$}: 
	
	 Fix $\epsilon:\Spec k\hookrightarrow S$, and let $Y$ be the closure of $\epsilon$. Let $a\in DA^{>}(k)$ 
	 . We want to find a $\bar a\in DA^1_c(U)$ for some neighborhood $U$ of $\Spec k$ in $Y$ such that for any $\delta:\Spec l\hookrightarrow S$, $\delta^!(a)\in DA^{>}(l)$ 
	 , while $\epsilon^*(a)=\epsilon^!(a)=a$. Now $a=\Sigma^{\infty}_{coh}(a')$ with $a'\in D^b(\mathcal M_1(k))$. 
	 
	 $\ast$ \emph{Additional claim:} To aid in the later part of the proof, we will additionally show that $\bar a\in D^{> t_{MM}^1}(U)$.
	 
	As before, we induct on the number of $i$ such that $H^i(a')\ne 0$.
		
	$\ast$ \emph{Base case}: The setup is the same with $a'=x'[i]$, but $i<0$. Then as before we have $\bar x'\in \mathcal M_1(U)$, such that the restriction $\epsilon^!(\bar x')=\epsilon^*(\bar x') = x'$ and we let $\bar a = \Sigma^\infty_{coh}(\bar x'[i])$. Fix $\delta: \Spec l\hookrightarrow U$, and let $W$ denote the closure of $\delta$ in $U$, that is $W=Y\cap U$. Let $\delta$ break up as the composition $\delta:\Spec l\overset{\delta_1}\hookrightarrow W \overset i\hookrightarrow U$. We want to show that $\delta^!(\bar a) := \delta_1^*\omega^1i^!(\bar a)$ is in $DA^{>}(l)$, whence $\bar a$ will be the required extension.
	
	Now note that $\bar a \in \Sigma^{\infty}_{coh}(\mathcal M_1(U))[i] \subset \mathbf{MM}^1(U)[i]\subset D^{>t_{MM}^1}(U)$ since $i<0$, using \cite[4.22]{lehalleur2015motivic}. By \cite[4.13]{lehalleur2015motivic} $\omega^1 i^!$ preserves $D^{>t_{MM}^1}(-)$ and hence it follows that $\omega^1i^!(\bar a)\in D^{>t_{MM}^1}(U)$. 
	
	Let $z\in DA^{\le}(l)$. By the additional claim in the continuity of $DA^{\le}(k)$ if follows that we can find a neighborhood of $\Spec l$, $\delta_1: \Spec l\overset{\delta_2}\hookrightarrow V \overset j\hookrightarrow W$ open and a $\bar z\in D^{\le t_{MM}^1}(V)$ such that $\delta_2^*(\bar z) = z$. In particular it follows that $\hom(\bar z, j^*\omega^1 i^!\bar a)=0$, since $j^*$ preserves $DA^{>t_{MM}^1}(k)$, using t-exactness of $j^*$ \cite[4.13]{lehalleur2015motivic}. We can even replace $V$ by any smaller open neighborhood of $\Spec l$, and hence using continuity, it follows that 
	\[
		\hom(z, \delta^! \bar a) = \hom(\delta_2^*\bar z, \delta_2^*j^*\omega^1i^!\bar a) = 0
	\]
	Since this is true for any $z\in DA^{\le}(l)$, we have that  $DA^{\le}(l)\perp \delta^!(\bar a)$, and hence by definition of a $t$-structure, $\delta^!(\bar a)\in DA^{>}(l)$. This completes the base case and we are done.
	
	\emph{Induction step}: This is ditto as in previous part (note that $\epsilon^*=\epsilon^!$ for $\epsilon$ generic).
	
	\underline{\emph{Verification}}: By what we have seen so far, using \ref{gluing:mainresult} gives a $t$ structure on $DA^1_c(S)$, let us denote it by $(D^{\le}(S), D^{>}(S))$ . We want to verify that this glued $t$-structure is same as the restriction of $t_{MM}^1(S)$, in particular, we have to show that $D^{\le}(S)\subset D^{\le t^1_{MM}}(S)$ and $D^{>}(S)\subset D^{> t^1_{MM}}(S)$. This is equivalent to showing that $D^{\le}(S)\perp D^{> t^1_{MM}}(S)$ and $D^{\le t^1_{MM}}(S)\perp D^{>}(S)$. We prove the claim by Noetherian induction on $S$, the base case being $S=\Spec k$, where it is clear since $t^1_{MM}(k)$ restricts to compact objects and this was the $t$ structure that we used for gluing -- see \cite[4.18]{lehalleur2015motivic} and the discussion preceding it.
	
	For the induction step, first let $\bar a\in D^{\le}(S)$ and $z\in D^{>t^1_{MM}}(S)$. Let $\epsilon:\Spec K\hookrightarrow S$ be a generic point of $S$. Then, by definition of a glued $t$-structure, $a=\epsilon^*\bar a\in DA^{\le}(K)$. By the additional claim in the continuity of $DA^{\le}(-)$, it follows that we can find an open neighborhood $\Spec K\overset{\epsilon_1}\hookrightarrow U\overset j\hookrightarrow S$ and an $a'\in D^{\le t^1_{MM}}$ such that $\epsilon_1^*a' = a = \epsilon^*\bar a$. By continuity, shrinking $U$ if necessary, we can assume that $a'=j^*(a)$. Since $j^*$ is exact for $t^1_{MM}$ by \cite[4.13]{lehalleur2015motivic}, it follows that $j^*\bar a\perp j^*z$. 
	
	Let $Z$ denote the complement of $U$. By construction of a glued $t$ structure, $i^*\bar a\in D^{\le}(Z)$. By \cite[4.13]{lehalleur2015motivic}, $\omega^1i^!z\in D^{>t^1_{MM}}(Z)$. Then $i^*\bar a \perp \omega^1i^!z$ by Noetherian induction. Therefore $\bar a \perp z$ by using localization triangle. 
	
	The claim for $D^{\le t^1_{MM}}(S)\perp D^{>}(S)$ is similar and this completes the verification.
\end{proof}

\bibliographystyle{alpha}
\bibliography{ehp1}{}	
\end{document}